\documentclass[reqno]{amsart}
\usepackage{fullpage}
\usepackage{array}
\usepackage{hyperref}
\usepackage{graphicx,color}
\usepackage{float}
\usepackage{upgreek}

\newcommand{\Fal}[1]{\ensuremath{F\langle{#1}\rangle}}
\newcommand{\Id}[1]{\ensuremath{\textnormal{Id}({#1})}}

\newcommand{\gen}[1]{\ensuremath{\langle{#1}\rangle}}

\newcommand{\Gstar}{\ensuremath{(G,\ast)}-}

\newcommand{\IdGstar}[1]{\ensuremath{\textnormal{Id}^{\sharp}({#1})}}

\newtheorem{theorem}{Theorem}[section]
\newtheorem{example}[theorem]{Example}

\newtheorem{lemma}[theorem]{Lemma}
\newtheorem{corollary}[theorem]{Corollary}
\newtheorem{remark}[theorem]{Remark}
\newtheorem{proposition}[theorem]{Proposition}
\newtheorem{definition}[theorem]{Definition}
\newcolumntype{C}[1]{>{\centering\let\newline\\\arraybackslash\hspace{0pt}}m{#1}}

\usepackage[utf8]{inputenc}
\usepackage{amsfonts}
\usepackage{amssymb}
\usepackage{verbatim}
\usepackage{mathrsfs}

\numberwithin{equation}{section}

\begin{document}
	\title{On the colength sequence of algebras with graded involution}
	
\author[W. Q. COTA, R. B. DOS SANTOS AND A. C. VIEIRA]{W. Q. COTA, R. B. DOS SANTOS AND A. C. VIEIRA$^*$}
	
	\dedicatory{Departamento de Matemática, Instituto de Ciências Exatas, Universidade Federal de Minas Gerais. \\ Avenida Antonio Carlos 6627, 31123-970, Belo Horizonte, Brazil}

	\thanks{\footnotesize 
 	$^*$Corresponding author.}

  \thanks{{\it E-mail addresses:} quaresmawesley@gmail.com (Cota),  rafaelsantos23@ufmg.br (dos Santos), anacris@ufmg.br (Vieira).}

  \thanks{The first author was partially supported by FAPEMIG and the second and third authors were partially supported by FAPEMIG and by CNPq.}

	
	
	\begin{abstract} 
    
    In recent years, many results have been established regarding classifications of varieties whose colength sequences are bounded by a fixed constant. In this work, we explore this theme in the setting of algebras endowed with a graded involution, called $(G,*)$-algebras. We give an explicit description of the decomposition of the $\gen{n}$-cocharacter for some important $(G,*)$-algebras $A$, for every $\gen{n}=(n_1, \ldots, n_{2t})$. For each algebra $A$, the $n$th colength is defined as the number of irreducible components that appear in these decompositions. Our aim is to classify varieties whose $n$th colengths are bounded by a fixed constant. 
	\end{abstract}
	
	\maketitle

    \section{Introduction}

Determining the set of polynomial identities of an algebra is a notoriously challenging problem. However, the theory developed by Regev and Kemer has played a crucial role in understanding the asymptotic behavior of these identities. Their approach relies on the analysis of numerical invariants associated with the algebra, such as the codimension and colength sequences.

    Recall that an associative algebra $A$ over a field $F$ of characteristic zero is a PI-algebra if $A$ admits a nontrivial polynomial identity, that is, a nonzero polynomial $f$ in the free associative algebra $F\langle X \rangle$ that vanishes under all evaluations of elements of $A$. The set $\textnormal{Id(A)}$ of all polynomial identities of $A$ is an ideal invariant under all endomorphisms of the free algebra, called the $T$-ideal of $A$.  Since \textit{char}$(F)$ is zero, the $T$-ideal $\Id{A}$ is generated by a finite set of multilinear identities. Therefore, we consider the space $P_n$ of multilinear polynomials in the first $n$ variables.
    
    An efficient method to study the growth of polynomial identities satisfied by $A$ was introduced by Regev \cite{RG}. His approach is based on analyzing the dimension of the quotient space 
	$$P_n(A)=\frac{P_n}{P_n \cap \textnormal{Id}(A)}$$ and defining the $n$th codimension of $A$ as $c_n(A)=\mbox{dim}_F P_n(A)$. The growth of polynomial identities of $A$ is completely determined by the behavior of its codimension sequence and has been extensively studied as a fundamental tool in the theory of PI-algebras. A key result in this setting was obtained by Kemer \cite{KM}, who proved that the codimension sequence of a PI-algebra either grows exponentially or is polynomially bounded.

In this work, we focus on another significant invariant: the colength sequence of $A$. It is important to emphasize that $P_n(A)$ has the structure of a $S_n$-module. This allows us to consider its character $\chi_n(A)$, called the $n$th cocharacter of $A$ and compute $l_n(A)$, the $n$th colength of $A$, which counts the number of $S_n$-modules appearing in the decomposition of $P_n(A)$ into irreducibles.

The codimension and colength sequences are closely related. In fact, Mishchenko, Regev and Zaicev \cite{Mishchenko} characterized algebras having polynomial growth of codimensions through the behavior of the colength sequence. They proved that $A$ has polynomial growth if and only if there exists a constant $k \geq 0$ such that $l_n(A) \leq k$, for all $n \geq 1$. Therefore, it is natural to investigate which varieties, among those with polynomial growth, have $n$th colength bounded by a fixed constant $k$. Motivated by this question, Giambruno and La Mattina presented in \cite{GLa, Daniela} the classification of varieties whose $n$th colength is bounded by $4$, for $n$ large enough.

    In recent years, the previous results have been extended to the class of algebras endowed with some additional structure (see, for instance, \cite{Cota, Cota2, MN, NV, Ana}).

 The main goal of this article is to study the so-called \Gstar algebras. More specifically, we consider $G$ a finite abelian group of order $t$ and $A$ a finite-dimensional algebra graded by $G$ and endowed with an involution $*$ that preserves the homogeneous components of the grading.
 In this setting, it is natural to investigate the corresponding $(G,*)$-identities and the sequences $c_n^{\sharp}(A)$ and $l_n^\sharp(A)$ of $(G,*)$-codimensions and $(G,*)$-colengths, respectively, in order to better understand the aspects of PI-algebras.

   In this article, we characterize the $(G,*)$-varieties of colength sequence bounded by $3$ by excluding certain algebras from the variety. As a consequence, we obtain a finite list of $(G,*)$-algebras generating all varieties with colength at most $4$. 
   

	\section{Graded involutions}
	
	Consider $G$ a multiplicative group with unit element $1$ and let $A$ be an associative algebra over a field $F$ of characteristic zero. We say that $A$ is a $G$-graded algebra if there exist subspaces $A_{g}$, ${g\in G}$, such that $A=\bigoplus_{g\in G} A_{g}$ and $A_{g}A_{h}\subseteq A_{gh}$, for all $g,h\in G.$ The subspace $A_g,\, g\in G$, is called homogeneous component of degree $g$. Moreover, the support of a $G$-graded algebra $A$ is defined as $\mbox{supp}(A)= \{g\in G \mid A_{g} \neq \{0\} \}.$
	
	Any algebra $A$ is a $G$-graded algebra endowed with trivial $G$-grading given by $A_1=A$ and $A_g=\{0\}$, for all $g\in G- \{1\}$.
    
    Let $UT_m$ be the algebra of upper triangular matrices of order $m$ and $\mathsf{h}=(h_1, \ldots, h_m) \in G^m$ be an arbitrary $m$-tuple of elements of $G$. We may use $\mathsf{h}$ to define a $G$-grading on $UT_m$ as follows: $(UT_m)_{g}= \mbox{span}_F\{e_{ij}\mid h_i^{-1}h_j=g\}, g\in G.$ This $G$-grading is referred to as the elementary $G$-grading induced by $\mathsf{h}$ and will also be considered in certain subalgebras of $UT_m$.
	
	If $H$ is a subgroup of $G$ and $A$ is an $H$-graded algebra then we may consider $A$ as a $G$-graded algebra by considering the homogeneous subspaces $A_h$, $h\in H$, and defining $A_{g}=\{0\}$, for all $g\in G- H$. We call this grading by $G$-grading induced by $H$.
	
	\begin{example} Let $g\in G$ be an element of prime order $p$. Denote by $C_p=\langle g \rangle $ the cyclic group of order $p$ and $FC_p$ to be the group algebra of $C_p$ over $F$. We consider the $G$-grading on $FC_p$ induced by $C_p$, given by $(FC_p)_{g^i}=F g^i$, for all $0\leq i \leq p-1$.\end{example}
	
	A subalgebra $B$ of a $G$-graded algebra $A$ is called a $G$-graded subalgebra of $A$ if it admits a decomposition $B=\bigoplus_{g\in G} (B \cap A_{g})_{g}$ as $G$-graded algebra. We refer to this $G$-grading as $G$-grading induced by $A$.
	
	
	Recall that an involution defined in an algebra $A$ is a linear map satisfying $(a^{*})^{*}=a$ and $(ab)^{*}=b^*a^*$, for all $ a,b\in A$, and an algebra endowed with an involution is called a $*$-algebra. In this work, we are interested in involutions compatible with the grading. We say that an involution $*$ defined in a $G$-graded algebra $A$ is a graded involution if $A_g^*= A_g$, for all $g\in G$. A $G$-graded algebra $A$ endowed with a graded involution $*$ is called a \Gstar algebra. 
    
	
	
	

    Observe that the existence of a graded involution on $A$ implies that $\mbox{supp}(A)$ is a commutative subset of $G$. Therefore, without loss of generality, in this paper we assume that $G$ is an abelian group.
	
	For instance, any commutative $G$-graded algebra is a $(G,*)$-algebra with trivial involution. In particular, we denote by $FC_p$ as the group algebra $FC_p$ with $G$-grading induced by $C_p$ and trivial involution. Moreover, any $*$-algebra is a $(G,*)$-algebra with trivial $G$-grading. As an example, consider $(FC_2)_{*}$ to be the group algebra $FC_2$ with trivial $G$-grading and involution given by $(a_1+a_2g)^{*}= a_1-a_2g$.
	
	\begin{example} Consider $G$ a group of even order and $g\in G$ an element of order $2$ generating a cyclic subgroup $C_2$ of $G$. Denote by $(FC_2)^{\sharp}$ the algebra $FC_2$ with $G$-grading induced by $C_2$ and involution given by $(a_1+a_2g)^{*}= a_1-a_2g$. \end{example}
	
	\begin{example} Denote by $M_{1,\rho}$ the algebra $M=F(e_{11}+e_{44})+F(e_{22}+e_{33})+ Fe_{12}+Fe_{34}\subset UT_4$ with trivial $G$-grading and reflection involution defined as $e_{ij}^{\rho}=e_{n-j+1,\, n-i+1}$. Also, for $g\in G- \{1\}$ denote by $M_{g,\rho}$ the algebra $M$ endowed with reflection involution and $G$-grading determined by
$$e_{11}+e_{44},e_{22}+e_{33}\in (M_{g,\rho})_{1}\mbox{ and }e_{12}, e_{34}\in (M_{g,\rho})_{g}.$$
	\end{example}
	
	If $*$ is a graded involution on $A$ then we can decompose 
	$$A=\underset{g\in G}{\bigoplus}(A_{g}^+ +  A_{g}^-)$$ into a direct sum of vector subspaces, where $A_{g}^+= \{a\in A_g \mid a^*=a\}$ is the homogeneous symmetric component of degree $g$ and $A_{g}^-= \{a\in A_g \mid a^*=-a\}$ is the homogeneous skew component of degree $g$. 

 From now on, we will consider $G$ a finite abelian group of order $t$. For all $g\in G$, consider $ X_g^{*}= \{x_{i,g}, x_{i,g}^{*}\mid g\in G, i\geq 1\}$ a countable set of variables and define $X= \bigcup_{g\in G} X_g^{*}$. Let $F\langle X\mid G, * \rangle$ be the free associative \Gstar algebra generated by $X$ over $F$, whose elements are called \Gstar polynomials. Consider $X^+={\bigcup_{g\in G} }X_g^+$ and $X^-={\bigcup_{g\in G} }X_g^-$, where $X_g^+= \{x_{i,g}^+:= x_{i,g}+x_{i,g}^*\mid g\in G, i\geq 1\}$ is the set of homogeneous symmetric variables of degree $g$ and $X_g^-= \{x_{i,g}^-=x_{i,g}-x_{i,g}^*\mid g\in G, i\geq 1\}$ is the set of homogeneous skew variables of degree $g$. Then, we have $\mathcal{F}:=\Fal{X\mid G,*}= \Fal{X^+\cup X^-}$. Now, consider $$\mathcal{F}_g=\mathrm{span}_F \{w_{i_1,g_{i_1}}\cdots w_{i_m,g_{i_m}}\mid g_{i_1 }\cdots g_{i_m}=g, w_{j,g_j}\in \{x_{j,g_j}^+,x_{j,g_j}^-\}\}$$ the space of elements that have homogeneous degree $g$ and notice that $\mathcal{F}= \bigoplus_{g\in G} \mathcal{F}_{g}$ has a structure of \Gstar algebra.   
	\begin{definition}
		A \Gstar polynomial $f\in \mathcal{F}$ is a \Gstar identity of a $(G,*)$-algebra $A$ if 
		$\lambda(f)=0$ for any evaluation $\lambda$ of the variables by elements of $A$ such that $\lambda(x_{i,g}^\varepsilon) \in A_g^\varepsilon$, for all $g\in G$, $\varepsilon\in \{+,-\}$ and $1\leq i\leq n$. In this case, we write $f\equiv 0$ on $A$.
	\end{definition}
	
	Let $\textnormal{Id}^{\sharp}(A)\subseteq \mathcal{F}$ be the set of all \Gstar identities of $A$. Notice that $\textnormal{Id}^{\sharp}(A)$ is an ideal invariant under all endomorphisms of $\mathcal{F}$ that preserve the grading and commute with the involution, which is called the $T_G^*$-ideal of $A$.
 
 	We say that two \Gstar algebras $A$ and $B$ are $T_G^*$-equivalent if $\textnormal{Id}^{\sharp} (A) = \textnormal{Id}^{\sharp}(B)$ and we write $A\sim_{T_G^*} B$. Moreover, we define $\mathcal{V}:=\textnormal{var}^{\sharp}(A)$ the $(G,*)$-variety generated by $A$, i.e., the class of all \Gstar algebras $B$ such that $\IdGstar{A}\subseteq \IdGstar{B}$.

			

  Since \textit{char}$(F)$ is zero then the $T_G^*$-ideal of $A$ is generated by its multilinear \Gstar polynomials. Thus, we consider $$P_n^{\sharp}=\mbox{span}_F\{w_{\sigma(1)} \cdots w_{\sigma(n)}\mid \sigma\in S_n,\; w_i\in \{x_{i,g}^+,x_{i,g}^-\},\; g\in G \}$$ the space of multilinear \Gstar polynomials of degree $n$. For $n \geq 1$, the $n$th $(G, *)$-codimension of a $(G, *)$-algebra $A$ is defined by
$$ c_n^{\sharp}(A) := \dim_F P_n^\sharp(A), \mbox{ where }P_n^\sharp(A)={P_n^{\sharp}}/{P_n^{\sharp} \cap \textnormal{Id}^{\sharp}(A)}. $$
	
    According to \cite{Lorena, RG}, we have that $A$ is a PI-algebra if and only its corresponding sequence of \Gstar codimensions is exponentially bounded. In particular, we are interested in $(G,*)$-varieties $\textnormal{var}^\sharp(A)$ of polynomial growth, that is, those satisfying $c_n^{\sharp}(A)\leq \alpha n^t$, for some constants $\alpha, t\geq 0$ and for all $n\geq 1$.

	

	Now we introduce some characterizations of $(G,*)$-varieties of polynomial growth. These results extend those presented in \cite{FR, GR} and can be found in \cite{Mara, Lorena}. 
	
	\begin{theorem}  \label{1}
		Let $\mathcal{V}$ be a variety generated by a finite-dimensional \Gstar algebra over a field $F$ of characteristic zero. Then the following conditions are equivalent
        \begin{enumerate}
           \item[1)] 
           $\mathcal{V}$ has polynomial growth. \item[2)]  for all prime $p$ dividing $|G|$ and for all $g \in G$, either $|G|$ is odd and $(FC_2)_{*}$, $FC_p$, $M_{g,\rho}$ $\notin \mathcal{V}$, or $|G|$ is even and $(FC_2)_{*}$, $(FC_2)^{\sharp}$, $FC_p$, $M_{g,\rho}$ $\notin \mathcal{V}$.
           \item[3)] $A\sim_{T_G^*} B_1 \oplus \cdots \oplus B_m$, where $B_1,$ $ \ldots, $ $B_m$ are finite-dimensional \Gstar algebras over $F$ with $\dim_F {B_i}/{J(B_i)}\leq 1$, for all $i=1, \ldots, m$, where $J(B_i)$ is the Jacobson radical of $B_i$.
           \end{enumerate} 
    
	\end{theorem}

 \section{\texorpdfstring{The $(G,*)$-colength}{The (G,*)-colength}}
	
	Recall that $G$ is a finite abelian group of order $t$ and consider $n=n_1+\cdots + n_{2t}$ a sum of $2t$ non-negative integers, which is denoted by $\langle n \rangle=(n_1, \ldots , n_{2t})$. Let $P_{\langle n \rangle}$ be the space of multilinear \Gstar polynomials where the first $n_1$ variables are symmetric of homogeneous degree $g_1=1$, the second $n_2$ variables are skew of homogeneous degree $1$, and so on until the penultimate $n_{2k-1}$ variables are symmetric of homogeneous degree $g_k$ and the last $n_{2k}$ variables are skew of homogeneous degree $g_k$. Define $$P_{\langle n \rangle}(A)= \frac{P_{\langle n \rangle}}{P_{\langle n \rangle} \cap \textnormal{Id}^{\sharp}(A)} $$ and consider $c_{\langle n \rangle}(A)= \dim_F P_{\langle n \rangle}(A)$ the $\langle n \rangle$-codimension of $A$. Observe that $$P_n^{\sharp} \cong \displaystyle \bigoplus_{\langle n \rangle } \displaystyle\binom{n}{ n_1, \ldots, n_{2t}} P_{\langle n \rangle},$$ where $\binom{n}{n_1, \ldots, n_{2t}}= \frac{n!}{n_1!\ldots n_{2t}!}$ is the multinomial coefficient, and so
	\begin{equation} \label{293}
		c_n^{\sharp}(A)= \underset{\langle n \rangle  }{\sum} \displaystyle\binom{n}{ n_1, \ldots, n_{2t} } c_{\langle n \rangle}(A). 
	\end{equation}
	
	Notice that there is a natural left action of $S_{\langle n \rangle}:= S_{n_1} \times \cdots \times S_{n_{2t}}$ on $P_{\langle n \rangle}$, where $S_{n_i}$ acts by permuting the corresponding variables associated to $n_i$, $1\leq i \leq 2t$. Since $P_{\langle n \rangle } \cap \textnormal{Id}^{\sharp}(A)$ is invariant under this action then $P_{\langle n \rangle }(A)$ inherits a structure of $S_{\langle n \rangle}$-module. Therefore, we can consider $\chi_{\langle n \rangle}(A)$ its character, called $\langle n \rangle$-cocharacter of $A$. It is known that the irreducibles $S_{\langle n \rangle}$-characters are outer tensor product of irreducible $S_{n_i}$-characters which are in one-to-one correspondence between partitions $\lambda_i \vdash n_i$. Hence, we consider $\chi_{\lambda_1} \otimes \cdots \otimes \chi_{\lambda_{2t}}$ the irreducible $S_{\langle n \rangle}$-character associated to a multipartition $\langle \lambda \rangle := ( \lambda_1, \ldots , \lambda_{2t})  \vdash \langle n \rangle$, where $\chi_{\lambda_i}$ is the irreducible $S_{n_i}$-character associated to $\lambda_i$. Moreover, its degree is given by $d_{\lambda_1} \cdots d_{\lambda _{2t}},$ where $d_{\lambda_i}$ is the degree of the $S_{n_i}$-character associated to $\lambda_i$ calculated by the hook formula \cite[Theorem 3.10.2]{Sagan}. Thus, by complete
	reducibility we may consider 
	\begin{equation} \label{77}
		\chi_{\langle n \rangle }(A)=\underset{\langle \lambda \rangle \vdash \langle n \rangle}{\sum} {m}_{\langle \lambda \rangle} \chi_{\lambda_1} \otimes \cdots \otimes \chi_{\lambda_{2t}}
	\end{equation} the decomposition of the $\langle n \rangle$-character of the space $P_{\langle n \rangle}(A)$ into irreducible, where ${m}_{\langle \lambda \rangle }$ denotes the multiplicity of the corresponding irreducible character. 
	For all possibilities $(n_{i_{1}},\ldots , n_{i_{2t}}),$ $ \ldots ,$ $ (n_{j_{1}},\ldots , n_{j_{2t}})$ of sums of $2t$ non-negative integers equal to $n$ we represent the set $\{\chi_{\langle n \rangle  }(A)\mid \langle n\rangle =(n_1, \ldots, n_{2t})  \}$ of all nonzero $\langle n \rangle $-cocharacters of $A$ by
	$$\mathcal{X}_n(A)=(\chi_{(n_{i_{1}},\ldots , n_{i_{2t}})}(A); \ldots ; \chi_{(n_{j_{1}},\ldots , n_{j_{2t}})}(A)).$$
	
	\begin{definition}
		Let $A$ be a \Gstar algebra with $\langle n \rangle$-cocharacter $\chi_{ \langle n\rangle }(A)$ as given in (\ref{77}). The $n$th \Gstar colength of $A$ is defined by
		$$l_n^{\sharp}(A)=\displaystyle \sum_{\langle n\rangle  }\sum\limits_{\displaystyle{
				\scriptsize{\begin{array}{c}
						\gen{\lambda}\vdash \gen{n}
		\end{array}}}}{m}_{\gen{\lambda}}.$$
	\end{definition}
	
In \cite{Mara2}, de Oliveira, dos Santos and Vieira  characterized the $(G,*)$-varieties of polynomial growth by the sequence of $(G,*)$-colength.

	\begin{theorem} \label{890}
		Let $A$ be a finite-dimensional $(G,*)$-algebra over a field of characteristic zero. Then, $A$ has polynomial growth if and only if there exists a constant $q\geq 0$ such that $l_n^{\sharp}(A)\leq q,$ for all $n\geq 1$.
	\end{theorem}

The theorem above provides a direct connection between codimension and colength sequences. We stress that classifying varieties of polynomial growth is not an easy task. Therefore, studying varieties through their colength sequence offers an alternative approach that may contribute to a better understanding of the varieties of polynomial growth.
	
	\begin{remark}\label{obs}
		Let $A$, $B$ and $A\oplus B$ be \Gstar algebras with the respective $\langle n\rangle$-cocharacter 
		$$ \chi_{\langle n\rangle }(A)=\displaystyle \sum_{\gen{\lambda}\vdash \gen{n}}
		{m}_{\gen{\lambda}}\chi_{\gen{\lambda}},\,\,\,\,\chi_{{\langle n\rangle }}(B)=\displaystyle \sum_{\gen{\lambda}\vdash \gen{n}}{m'}_{\gen{\lambda}}\chi_{\gen{\lambda}} \mbox{ and } \chi_{\langle n\rangle }(A\oplus B)=\displaystyle \sum_{\gen{\lambda}\vdash \gen{n}}\widetilde{m}_{\gen{\lambda}}\chi_{\gen{\lambda}}.$$

        \begin{enumerate}
            \item[1)]  If $B \in \textnormal{var}^{\sharp}(A)$ then ${m'}_{\gen{\lambda}} \leq {m}_{\gen{\lambda}}$, for all $\gen{n}$ and $\gen{\lambda} \vdash \gen{n}$.Therefore, $l_n^{\sharp}(B) \leq l_n^{\sharp}(A)$.
            
            \item[2)] For all $\gen{n}$ and $\gen{\lambda} \vdash \gen{n}$ we have $\widetilde{m}_{\gen{\lambda}} \leq {m}_{\gen{\lambda}}+{m'}_{\gen{\lambda}}$.

            \item[3)] If $B\in \textnormal{var}^{\sharp}(A)$ and $l_n^{\sharp}(A)=l_n^{\sharp}(B)$ then, by item $1)$, we must have ${m'}_{\gen{\lambda}} = {m}_{\gen{\lambda}}$, for all $\gen{n}$ and $\gen{\lambda} \vdash \gen{n}$. Thus, $c_n^{\sharp}(A)=c_n^{\sharp}(B)$ and so $P_n^{\sharp}\cap \textnormal{Id}^{\sharp}(A)=P_n^{\sharp}\cap \textnormal{Id}^{\sharp}(B)$, for all $n$. Therefore, we have $\textnormal{var}^{\sharp}(A)=\textnormal{var}^{\sharp}(B)$.

        \end{enumerate}

	\end{remark}


    There is a well-established method for computing the multiplicities $m_{\langle \lambda \rangle}$, which is based on the representation theory of the general linear group $GL_m$ in the context of $(G,*)$-algebras. We assume that the reader is familiar with the results concerning the generators of irreducible $GL_m$-modules, known as highest weight vectors. A detailed exposition of this theory will be omitted here (for further information, see \cite[Section 12.4]{Drensky}). 
    
    The construction of highest weight vectors in the context of $(G,*)$-algebras is naturally derived from the case of $(\mathbb{Z}_2,*)$-algebras, as developed in \cite[Section~3]{NV}. Here we recall the following result.

    \begin{proposition} \label{calmult}
       The multiplicity ${m}_{\gen{\lambda}}$ is nonzero if and only if there exists a multitableau $T_{\gen{\lambda}}$ associated to $\langle \lambda \rangle$ such that the corresponding highest weight vectors $f_{T_{\gen{\lambda}}} \notin \textnormal{Id}^{\sharp}(A)$. Furthermore, ${m}_{\gen{\lambda}}$ is equal to the maximum number of linearly independent highest weight vectors $f_{T_{\gen{\lambda}}}$ modulo $\textnormal{Id}^\sharp (A)$.
    \end{proposition}

	In the next section, it will be convenient to use the notation $\langle \lambda \rangle = ((\lambda_{i_1})_{{g_{i_1}^+}}, (\lambda_{i_2})_{g_{i_2}^-}, \ldots  )$, where $(\lambda_{i_1})_{{g_{i_1}^+}}$ means that $(\lambda_{i_1})$ is the partition of $n_{2i_1-1}$, $(\lambda_{i_2})_{g_{i_2}^-}$ means that $(\lambda_{i_2})$ is a partition of $n_{2i_2}$ and so on. Also, we omit the empty partitions in this notation. 
	
	\section{\texorpdfstring{Constructing $(G,*)$-algebras}{Constructing (G,*)-algebras}}
	
	In this section, we exhibit the decomposition of the $\langle n \rangle$-cocharacter of particular \Gstar algebras.

    We start considering, for $k\geq 2$,  the subalgebra $\mathcal{G}_k=\langle 1,e_1,\ldots, e_k \mid e_ie_j=-e_je_i  \rangle$  of the infinite-dimensional Grassmann algebra. For this algebra, we consider the involutions $\psi, \tau$ and $\gamma$ defined, respectively, by
	\begin{align*}
		\psi( e_i)=e_i,\,\,\,\,\, \tau( e_i)=-e_i,\,\,\,\,\, \gamma( e_i)=(-1)^ie_i,\mbox{ for all }i=1,\ldots, k.
	\end{align*}

\begin{definition}
    For $g,h\in G$, we denote by $\mathcal{G}_{2,*}^{g,h}$ as the algebra $\mathcal{G}_2$ with involution $*\in \{\psi, \tau, \gamma\}$ and the $G$-grading determined by $$ 1\in (\mathcal{G}_{2,*}^{g,h})_1,\, e_1\in (\mathcal{G}_{2,*}^{g,h})_g ,\,  e_2\in (\mathcal{G}_{2,*}^{g,h})_h\mbox{ and } \,e_1e_2\in (\mathcal{G}_{2,*}^{g,h})_{gh} .$$ Moreover, we consider the algebra $\mathcal{G}_{3,\tau}^{1,1}$ as the algebra $\mathcal{G}_{3}$ with trivial grading and involution $\tau$.
\end{definition}

By \cite{MN, NV} we have $l_{n}^{\sharp}(\mathcal{G}_{2,\tau}^{1,1})=3$ and $l_n^{\sharp}(\mathcal{G}_{2,*}^{1,g})=l_{n}^{\sharp}(\mathcal{G}_{3,\tau}^{1,1})=4$, for all $*\in \{\tau, \gamma\}$ and $g\in G-\{1\}$. Also, if $|G|$ is even and $h\in G$ with $|h|=2$ then $l_{n}^{\sharp}(\mathcal{G}_{2,*}^{h,h})=4,$ for all $* \in \{\psi, \tau, \gamma\}$.

  From now on, we use the notation $x_{i,g}$ to indicate any variable in the set $\{x_{i,g}^+,x_{i,g}^-\}$. Moreover, for each $(G,*)$-algebra $A$, the
meaning will be clear from context, $r$ will run through all the elements in the set $G-\mbox{supp}(A)$.

	\begin{lemma} \label{18} For all $g\in G$ with $|g|> 2$ we have
		
		\begin{enumerate}
			
			\item[1)] $\textnormal{Id}^{\sharp}(\mathcal{G}_{2,\tau}^{g,g})=\langle x_{1,1}^-, x_{1,g}^+, x_{1,g^2}^+, x_{1,r}  \rangle_{T_G^*}.$
			
			\item[2)] $\textnormal{Id}^{\sharp}(\mathcal{G}_{2,\psi}^{g,g})=\langle x_{1,1}^-, x_{1,g}^-, x_{1,g^2}^+, x_{1,r}  \rangle_{T_G^*}.$
			
			\item[3)] $\textnormal{Id}^{\sharp}(\mathcal{G}_{2,\gamma}^{g,g})=\langle x_{1,1}^-, x_{1,g^2}^-, x_{1,g}^+x_{2,g}^+, x_{1,g}^-x_{2,g}^-,[{x_{1,1}^+},x_{2,g}^+],$ $[{x_{1,1}^+},x_{2,g}^-],$ $x_{1,r}\rangle_{T_G^*}.$
			
			\item[4)] $\mathcal{X}_n(\mathcal{G}_{2,\tau}^{g,g})= (\chi_{((n)_{1^+})}; \chi_{((n-1)_{1^+}, (1)_{g^-})};\chi_{((n-1)_{1^+}, (1)_{{g^2}^-})};\chi_{((n-2)_{1^+}, (1^2)_{g^-})} ).$

			\item[5)] $\mathcal{X}_n(\mathcal{G}_{2,\psi}^{g,g})=( \chi_{((n)_{1^+})}; \chi_{((n-1)_{1^+} , (1)_{g^+})};\chi_{((n-1)_{1^+}, (1)_{{g^2}^-})};\chi_{((n-2)_{1^+}, (1^2)_{g^+})}).$

			\item[6)] $\mathcal{X}_n(\mathcal{G}_{2,\gamma}^{g,g})= (\chi_{((n)_{1^+})}; \chi_{((n-1)_{1^+} , (1)_{g^+})};\chi_{((n-1)_{1^+} , (1)_{g^-})};\chi_{((n-1)_{1^+}, (1)_{{g^2}^+})};\chi_{((n-2)_{1^+}, (1)_{g^+} , (1)_{g^-})}).$
			\item[7)] $c_n^{\sharp}(\mathcal{G}_{2,\iota}^{g,g})=1+2n+\displaystyle\binom{n}{2}$ \;and\; $c_n^{\sharp}(\mathcal{G}_{2,\gamma}^{g,g})=1+3n+2\displaystyle\binom{n}{2}$, for all $\iota \in \{\tau, \psi\}$. 
   
   \item[8)] $l_{n}^{\sharp}(\mathcal{G}_{2,\iota}^{g,g})=4$ \;and\; $l_{n}^{\sharp}(\mathcal{G}_{2,\gamma}^{g,g})=5,$ for all $\iota \in \{\tau, \psi\}.$
		\end{enumerate}
	\end{lemma}

\begin{proof}
    We begin by proving item {$1)$}. Consider 
\[
I=\langle x_{1,1}^-,\, x_{1,g}^+,\, x_{1,g^2}^+,\, x_{1,r} \rangle_{T_G^*}.
\]  
It is straightforward to check that $I \subseteq \textnormal{Id}^{\sharp}(\mathcal{G}_{2,\tau}^{g,g}).$ Let $f \in P_n^{\sharp}\cap \textnormal{Id}^{\sharp}(\mathcal{G}_{2,\tau}^{g,g})$. Observe that the $(G,*)$-polynomials  
\[
x_{1,g^2}^-x_{2,g^2}^-, \quad 
x_{1,g^2}^-x_{2,g}^-, \quad 
[{x_{1,1}^+},x_{2,1}^+], \quad 
[{x_{1,1}^+},x_{1,g^2}^-], \quad 
x_{1,g}^- \circ x_{2,g}^-, \quad 
x_{1,g}^-x_{2,g}^-x_{3,g}^-,
\]  
where $x_{1,g}^- \circ x_{2,g}^- = x_{1,g}^-x_{2,g}^- + x_{2,g}^-x_{1,g}^-$ is the Jordan product, are consequences of the generators of $I$. Thus, we can reduce $f$ modulo $I$ and rewrite it as a linear combination of the $(G,*)$-polynomials
\begin{equation} \label{polyn}
    {x_{1,1}^+}\cdots x_{n,1}^+, \quad  
    {x_{1,1}^+}\cdots \widehat{x_{i,1}^+}\cdots x_{n,1}^+x_{l,g}^-, \quad 
    {x_{1,1}^+}\cdots \widehat{x_{i,1}^+}\cdots {\widehat{x_{j,1}^+}}\cdots x_{n,1}^+x_{i,g}^-x_{j,g}^-, \quad 
    {x_{1,1}^+}\cdots \widehat{x_{i,1}^+}\cdots x_{n,1}^+x_{l,g^2}^-,
\end{equation}
where $1\leq i<j \leq n$ and $\widehat{x_{i,g}^\epsilon}$ indicates omission of the variable $x_{i,g}^\epsilon$.  

Since $\textnormal{Id}^{\sharp}(\mathcal{G}_{2,\tau}^{g,g})$ is generated by its multihomogeneous $(G,*)$-identities, we may assume, without loss of generality, that $f$ is one of the following polynomials:
\[
f= \alpha  {x_{1,1}^+}\cdots x_{n,1}^+, \quad 
f= \beta {x_{1,1}^+}\cdots x_{n-1,1}^+x_{n,g}^-, \quad 
f= \gamma {x_{1,1}^+} \cdots x_{n-2,1}^+x_{n-1,g}^-x_{n,g}^-, \quad 
f= \delta {x_{1,1}^+}\cdots x_{n-1,1}^+x_{n,g^2}^-.
\]  

Considering the evaluation 
\[
x_{i,1}^+\mapsto 1, \ \mbox{ for all } i, \quad 
x_{n-1,g}^-\mapsto e_1, \quad 
x_{n,g}^-\mapsto e_2, \quad 
x_{n,g^2}^-\mapsto e_1e_2,
\]  
we obtain $\alpha=\beta=\gamma=\delta=0$. Hence $f\in I$, proving that $I=\textnormal{Id}^{\sharp}(\mathcal{G}_{2,\tau}^{g,g}).$  

Moreover, the above argument shows that no nonzero linear combination of the $(G,*)$-polynomials in (\ref{polyn}) belongs to $\textnormal{Id}^{\sharp}(\mathcal{G}_{2,\tau}^{g,g})$. Thus, these polynomials are linearly independent modulo the $T_G^*$-ideal of $\mathcal{G}_{2,\tau}^{g,g}$. Since they generate $P_n^{\sharp}$ modulo $P_n^{\sharp}\cap \textnormal{Id}^{\sharp}(\mathcal{G}_{2,\tau}^{g,g})$, the polynomials in (\ref{polyn}) form a basis for $P_n^{\sharp}$ modulo $\textnormal{Id}^{\sharp}(\mathcal{G}_{2,\tau}^{g,g})$. Consequently,  
\[
c_n^{\sharp}(\mathcal{G}_{2,\tau}^{g,g}) = 1+2n+\binom{n}{2}.
\]  

Items {2)} and {3)} follow by analogous arguments.  

For item {4)}, we use the decomposition given in (\ref{293}) and (\ref{77}), which yields
\[
c_n^{\sharp}(\mathcal{G}_{2,\tau}^{g,g})= d_{(n)} 
+ \binom{n}{n-1} d_{(n-1)}d_{(1)} 
+ \binom{n}{n-1} d_{(n-1)}d_{(1)} 
+ \binom{n}{n-2} d_{(n-2)}d_{(1^2)}.
\]  

It remains to show that
\[
m_{\langle \lambda_1 \rangle}\geq 1,\quad 
m_{\langle \lambda_2 \rangle}\geq 1,\quad 
m_{\langle \lambda_3 \rangle}\geq 1,\quad 
m_{\langle \lambda_4 \rangle}\geq 1,\mbox{ where }
\]  
\[
\gen{\lambda_1}= ((n)_{1^+}),\quad 
\gen{\lambda_2}= ((n-1)_{1^+}, (1)_{g^-}),\quad 
\gen{\lambda_3}= ((n-1)_{1^+}, (1)_{g^2}^-),\quad 
\gen{\lambda_4}= ((n-2)_{1^+}, (1^2)_{g^-}).
\]  

By Proposition~\ref{calmult}, the multiplicity $m_{\langle \lambda \rangle}$ is equal to the maximal number of linearly independent highest weight vectors modulo $\textnormal{Id}^\sharp(\mathcal{G}_{2,\tau}^{g,g})$ corresponding to $\langle \lambda \rangle$. Consider
\[
f_{T_{\gen{\lambda_1}}}=(x_{1,1}^+)^n, \quad 
f_{T_{\gen{\lambda_2}}}=(x_{1,1}^+)^{n-1}x_{1,g}^-, \quad 
f_{T_{\gen{\lambda_3}}}=(x_{1,1}^+)^{n-1}x_{1,g^2}^-, \quad 
f_{T_{\gen{\lambda_4}}}=(x_{1,1}^+)^{n-2}[x_{1,g}^-,x_{2,g}^-].
\]  

Considering the evaluation $x_{1,1}^+\mapsto 1, \quad 
x_{1,g}^-\mapsto e_1, \quad 
x_{2,g}^-\mapsto e_2, \quad 
x_{1,g^2}^-\mapsto e_1e_2,$ we see that $f_{T_{\gen{\lambda_1}}},$ $ f_{T_{\gen{\lambda_2}}},$ $ f_{T_{\gen{\lambda_3}}},$ $f_{T_{\gen{\lambda_4}}}\notin \textnormal{Id}^{\sharp}(\mathcal{G}_{2,\tau}^{g,g}).$  Therefore, $m_{\langle \lambda_i\rangle}\geq 1$, for all $i=1,2,3$, then the result follows.  

The remaining items can be proved analogously. 
\end{proof}

The proofs of the following propositions proceed by arguments entirely analogous to that of the previous proposition, and will thus be omitted for brevity.

	\begin{lemma} Consider $g\in G- \{1\}$ such that $|g|>2$. Then,
		\begin{enumerate}
			
			\item[1)] $\textnormal{Id}^{\sharp}(\mathcal{G}_{2,\tau}^{g,g^{-1}})=\langle x_{1,g}^+, x_{1,g^{-1}}^+, [{x_{1,1}^+},x_{2,1}^+], [{x_{1,1}^+},x_{2,1}^-], x_{1,1}^-x_{2,1}^-, x_{1,1}^-x_{2,g}^-,$ $x_{1,1}^-x_{2,g^{-1}}^-,$  $x_{1,g}^-\circ x_{2,g^{-1}}^-,$ $x_{1,r} \rangle_{T_G^*}.$
			
			\item[2)] $\textnormal{Id}^{\sharp}(\mathcal{G}_{2,\psi}^{g,g^{-1}})=\langle x_{1,g}^-,x_{1,g^{-1}}^-, [{x_{1,1}^+},x_{2,1}^+], [{x_{1,1}^+},x_{2,1}^-], x_{1,1}^-x_{2,1}^-,x_{1,1}^-x_{2,g}^+, x_{1,1}^-x_{2,g^{-1}}^+,$ $x_{1,g}^+\circ x_{2,g^{-1}}^+,$ $ x_{1,r}\rangle_{T_G^*}.$
			
			\item[3)] $\textnormal{Id}^{\sharp}(\mathcal{G}_{2,\gamma}^{g,g^{-1}})=\langle x_{1,1}^-, x_{1,g}^-, x_{1,g^{-1}}^+, x_{1,r} \rangle_{T_G^*} $.
			
			\item[4)] $\mathcal{X}_n(\mathcal{G}_{2,\tau}^{g,g^{-1}})=(\chi_{((n)_{1^+})}; \chi_{((n-1)_{1^+}, (1)_{1^-})};\chi_{((n-1)_{1^+}, (1)_{g^-})};\chi_{((n-1)_{1^+}, (1)_{(g^{-1})^-})};$ $\chi_{((n-2)_{1^+}, (1)_{g^-}, (1)_{(g^{-1})^-})})$.
			
			\item[5)] $\mathcal{X}_n(\mathcal{G}_{2,\psi}^{g,g^{-1}})=(\chi_{((n)_{1^+})}; \chi_{((n-1)_{1^+}, (1)_{1^-})};\chi_{((n-1)_{1^+}, (1)_{g^+})};\chi_{((n-1)_{1^+}, (1)_{(g^{-1})^+})};\chi_{((n-2)_{1^+},  (1)_{g^+}, (1)_{(g^{-1})^+})})$.
			
			\item[6)] $\mathcal{X}_n(\mathcal{G}_{2,\gamma}^{g,g^{-1}})=(\chi_{((n)_{1^+})}; \chi_{((n-1)_{1^+}, (1)_{g^+})};\chi_{((n-1)_{1^+}, (1)_{(g^{-1})^+})};\chi_{((n-2)_{1^+}, (1)_{g^+}, (1)_{(g^{-1})^-} )})$.
			
			\item[7)] $c_n^{\sharp}(\mathcal{G}_{2,\iota}^{g,g^{-1}})=1+3n+2\displaystyle\binom{n}{2}$ \;and\; $c_n^{\sharp}(\mathcal{G}_{2,\gamma}^{g,g^{-1}})=1+2n+2\displaystyle\binom{n}{2}$,  for all  $\iota \in \{\tau, \psi\}$. 
   
   \item[8)] $l_{n}^{\sharp}(\mathcal{G}_{2,\gamma}^{g,g^{-1}})=4$ \;and\; $l_{n}^{\sharp}(\mathcal{G}_{2,\iota}^{g,g^{-1}})=5,$  for all  $\iota \in \{\tau, \psi\}$.
			
		\end{enumerate}
	\end{lemma}  
  
	\begin{lemma} Consider distinct elements $g,h\in G- \{1\}$ such that $gh\neq 1$. Then,
		\begin{enumerate}
			\item[1)] $\textnormal{Id}^{\sharp}(\mathcal{G}_{2,\tau}^{g,h})=\langle x_{1,1}^-, x_{1,g}^+, x_{1,h}^+$, $x_{1,gh}^+, x_{1,g}^-x_{2,g}^-,$ $x_{1,g}^-x_{2,gh}^-$, $x_{1,h}^-x_{2,h}^-,$ $x_{1,h}^-x_{2,gh}^-,$ $x_{1,gh}^-x_{2,gh}^-, x_{1,r}\rangle_{T_G^*}.$
			
			\item[2)] $\textnormal{Id}^{\sharp}(\mathcal{G}_{2,\psi}^{g,h})=\langle  x_{1,1}^-,x_{1,g}^-, x_{1,h}^-, x_{1,gh}^+,$ $x_{1,g}^+x_{2,g}^+$, $x_{1,g}^+x_{2,gh}^-,$ $x_{1,h}^+x_{2,h}^+,$ $x_{1,h}^+x_{2,gh}^-,$ $x_{1,gh}^-x_{2,gh}^-, x_{1,r}\rangle_{T_G^*} $.
			
			\item[3)] $\textnormal{Id}^{\sharp}(\mathcal{G}_{2,\gamma}^{g,h})=\langle x_{1,1}^-, {x_{1,g}^+, x_{1,h}^-},x_{1,gh}^-, x_{1,g}^+x_{2,g}^+,$ $x_{1,g}^+x_{2,gh}^+,$ $x_{1,h}^-x_{2,h}^-,$ $x_{1,h}^-x_{2,gh}^+$, $x_{1,gh}^+x_{2,gh}^+, x_{1,r}\rangle_{T_G^*} $. 
			
			\item[4)] $\mathcal{X}_n(\mathcal{G}_{2,\tau}^{g,h})=(\chi_{((n)_{1^+})}; \chi_{((n-1)_{1^+}, (1)_{g^-})};\chi_{((n-1)_{1^+},(1)_{h^-})};\chi_{((n-1)_{1^+},  (1)_{gh^-} )}; \chi_{((n-2)_{1^+}, (1)_{g^-}, (1)_{h^-})}).$
			
			\item[5)] $\mathcal{X}_n(\mathcal{G}_{2,\psi}^{g,h})=(\chi_{((n)_{1^+})}; \chi_{((n-1)_{1^+}, (1)_{g^+})};\chi_{((n-1)_{1^+}, (1)_{h^+})};\chi_{((n-1)_{1^+},  (1)_{gh^-} )};\chi_{((n-2)_{1^+}, (1)_{g^+}, (1)_{h^+})}).$
			
			\item[6)] $\mathcal{X}_n(\mathcal{G}_{2,\gamma}^{g,h})=(\chi_{((n)_{1^+})}; {\chi_{((n-1)_{1^+},  (1)_{g^-})};\chi_{((n-1)_{1^+}, (1)_{h^+})}};\chi_{((n-1)_{1^+},  (1)_{gh^+} )};\chi_{((n-2)_{1^+},  (1)_{g^+}, (1)_{h^-})}).$
			
			\item[7)] $c_n^{\sharp}(\mathcal{G}_{2,\iota}^{g,h})=1+3n+2\displaystyle\binom{n}{2}$ \;and\; $l_{n}^{\sharp}(\mathcal{G}_{2,\iota}^{g,h})=5,$
			for all $\iota \in \{\tau, \gamma, \psi\}$.
			
		\end{enumerate}
	\end{lemma}
	
	For $k \geq 2$ denote the $k\times k$ identity matrix by $I_{k}$ and let $E_1 = \sum\limits_{i = 1}^{k-1} e_{i,i+1} \in UT_{k}$. Define the commutative subalgebra of $UT_k$ given by
	$$C_{k}= \{\alpha I_k + \underset{1\leq i<k}{\sum} \alpha_i E_1^i\mid \alpha, \alpha_i \in F\}$$
 and consider the following involution on $C_k$
 
 \begin{equation}\label{invCk}
     (\alpha I_{k} + \underset{1\leq i<k}{\sum} \alpha_i E_1^i)^*= \alpha I_{k} + \underset{1\leq i<k}{\sum} (-1)^i\alpha_i E_1^i.
 \end{equation}
	
 For each $ g \in G- \{1\}$ we define the following \Gstar algebras
		\begin{enumerate}
			\item[1)] $C_{k,*}$ is the algebra $C_k$ endowed with trivial $G$-grading and involution given in (\ref{invCk});
			\item[2)] $C_{k}^\mathsf{g}$ is the algebra $C_k$ endowed with $G$-grading induced by the $k$-tuple $\mathsf{g}=(1,g,g^2,$ $\ldots , g^{k-1})\in G^k$ and trivial involution;
			\item[3)]$C_{k,*}^\mathsf{g}$ is the algebra $C_k$ endowed with  $G$-grading induced by the $k$-tuple $\mathsf{g}=(1,g,g^2,$ $\ldots , g^{k-1})\in G^k$ and the involution given in (\ref{invCk}).
		\end{enumerate}

The next lemma can be consulted in \cite{RN, NV}.

	\begin{lemma} For all $k\geq 2$ and $g,h\in G-\{1\}$ with $|h|=2$ we have $l_{n}^{\sharp}(C_{2}^\mathsf{g})=l_{n}^{\sharp}(C_{2,*}^\mathsf{g})=2$ and $l_{n}^{\sharp}(C_{k,*})=l_{n}^{\sharp}(C_{k,*}^\mathsf{h})=l_{n}^{\sharp}(C_{k}^\mathsf{h})=k.$

	\end{lemma}

	For $g\in G- \{1\}$ we consider $\mathcal{D}^g=\{C_{2,*}, C_2^\mathsf{g}, C_{2,*}^\mathsf{g}\}$ and denote $\mathcal{D}=\underset{g\in G- \{1\}}{\bigcup} \mathcal{D}^g.$ Also we consider the following sets of \Gstar algebras
\begin{equation} \label{comm}
\mathcal{D}_1=\{D_1\oplus D_2\mid   D_i\in \mathcal{D}, D_1\neq D_2\}\mbox{ and } 
\end{equation} $$\newline \mathcal{D}_2=\{D_1\oplus D_2\oplus D_3\mid   D_i\in \mathcal{D}, D_i\neq D_j \mbox{ for }i\neq j\}.    $$

	\begin{remark}
		For the previous algebras we have $l_n^{\sharp}(S_1)=3$ and $l_n^{\sharp}(S_2)=l_n^{\sharp}(C_{3,*} \oplus K)= l_n^{\sharp}(\mathcal{G}_{2,\tau}^{1,1} \oplus K)=4$, for all $S_1\in \mathcal{D}_1$, $S_2\in \mathcal{D}_2$ and $K\in \mathcal{D}- \{C_{2,*}\}.$ Moreover, if $|G|$ is even and $g\in G$ with $|g|=2$ then $l_n^{\sharp}(C_{3}^\mathsf{g}\oplus T_1)=l_n^{\sharp}(C_{3,*}^\mathsf{g}\oplus T_2)=4$, for all $T_1,T_2\in \mathcal{D}$ with $T_1\neq C_{2}^\mathsf{g}$ and $T_2\neq C_{2,*}^\mathsf{g}.$
	\end{remark}

	

Now, we consider the following lemma. 
 
	\begin{lemma}   Consider $g\in G$ with $|g|>2$ and $h\in \{1,g,g^2\}$. Then,
		\begin{enumerate}
			\item[1)] $\textnormal{Id}^{\sharp}(C_{3}^\mathsf{g})=\langle x_{1,h}^-, x_{1,g}^+x_{2,g}^+x_{3,g}^+, x_{1,g}^+x_{2,g^2}^+, x_{1,g^2}^+x_{2,g^2}^+ ,x_{1,r} \rangle_{T_G^*}$.
			\item[2)] $\textnormal{Id}^{\sharp}(C_{3,*}^\mathsf{g})=\langle x_{1,1}^-,x_{1,g^2}^-, x_{1,g}^+,[x_{1,h},x_{2,r}], x_{1,g^2}^+x_{2,g^2}^+,x_{1,g}^-x_{2,g^2}^+, x_{1,r}\rangle_{T_G^*}$. 
		
 \item[3)]   $\mathcal{X}_n(C_{3}^\mathsf{g})= (\chi_{((n)_{1^+})}; \chi_{((n-1)_{1^+},  (1)_{g^+})}; \chi_{((n-2)_{1^+},   (2)_{g^+})}; \chi_{((n-1)_{1^+} ,  (1)_{(g^2)^+})}).$
			\item[4)] $\mathcal{X}_n(C_{3,*}^\mathsf{g})=(\chi_{((n)_{1^+})}; \chi_{((n-1)_{1^+},   (1)_{g^-})}; \chi_{((n-2)_{1^+}, (2)_{g^+})}; \chi_{((n-1)_{1^+},  (1)_{{g^2}^+})}).$

   \item[5)] $c_n^{\sharp}(C_{3}^\mathsf{g})= c_n^{\sharp}(C_{3,*}^\mathsf{g})=1+2n+\displaystyle\binom{n}{2}$ \;and\; $l_{n}^{\sharp}(C_{3}^\mathsf{g})=l_{n}^{\sharp}(C_{3,*}^\mathsf{g})=4.$ 
		\end{enumerate}
	\end{lemma}
	
	\begin{proof}
In order to prove item {1)}, consider 
\[
I=\langle x_{1,h}^-,\, x_{1,g}^+x_{2,g}^+x_{3,g}^+,\, x_{1,g}^+x_{2,g^2}^+,\, x_{1,g^2}^+x_{2,g^2}^+,\, x_{1,r} \rangle_{T_G^*}.
\] 
It is immediate that $I \subseteq \textnormal{Id}^{\sharp}(C_{3}^\mathsf{g}).$ To prove the opposite inclusion, let $f \in P_n^\sharp \cap \textnormal{Id}^{\sharp}(C_{3}^\mathsf{g}).$ After reducing $f$ modulo $I$, we may rewrite $f$ as a linear combination of the polynomials
\begin{equation} \label{algc3g}
    {x_{1,1}^+}\cdots x_{n,1}^+, \quad 
    {x_{1,1}^+}\cdots \widehat{x_{i,1}^+}\cdots x_{n,1}^+x_{i,g}^+, \quad 
    {x_{1,1}^+}\cdots \widehat{x_{i,1}^+}\cdots x_{n,1}^+x_{i,g^2}^+, \quad 
    {x_{1,1}^+}\cdots \widehat{x_{i,1}^+}\cdots {\widehat{x_{j,1}^+}}\cdots x_{n,1}^+x_{i,g}^+x_{j,g}^+,
\end{equation}
where $1\leq i<j \leq n$. By the multihomogeneity of $T_G^*$-ideals, without loss of generality we may assume that $f$ is one of the following polynomials:  
\[
 \alpha {x_{1,1}^+}\cdots x_{n,1}^+, \quad 
\beta {x_{1,1}^+}\cdots x_{n-1,1}^+x_{i,g}^+, \quad 
\gamma {x_{1,1}^+}\cdots x_{n-1,1}^+x_{j,g^2}^+, \quad \mbox{ or }\quad 
\delta {x_{1,1}^+}\cdots x_{n-2,1}^+x_{p,g}^+x_{q,g}^+.
\]  

Consider the evaluation
\[
x_{l,1}^+\mapsto I_3, \mbox{ for all } l, 
\qquad x_{i,g}^+\mapsto E_1, 
\qquad x_{j,g^2}^+\mapsto E_1^2, 
\qquad x_{p,g}^+,  x_{q,g}^+\mapsto E_1,
\] we obtain $\alpha=\beta=\gamma=\delta =0$. Hence, $I= \textnormal{Id}^{\sharp}(C_{3}^\mathsf{g}).$ 

The argument above shows that the polynomials in (\ref{algc3g}) form a basis of $P_n^{\sharp}$ modulo $P_n^{\sharp}\cap \textnormal{Id}^{\sharp}(C_{3}^\mathsf{g})$. Therefore,
\[
c_n^{\sharp}(C_{3}^\mathsf{g})=1+2n+\binom{n}{2}.
\]  

For item {3)}, by (\ref{293}) and (\ref{77}) we obtain
\[
c_n^{\sharp}(C_{3}^\mathsf{g})= d_{(n)} 
+ \binom{n}{n-1}d_{(n-1)}d_{(1)} 
+ \binom{n}{n-2}d_{(n-2)}d_{(2)} 
+ \binom{n}{n-1}d_{(n-1)}d_{(1)}.
\]  
It remains to show that  $m_{
\langle \lambda_i \rangle}\geq 1$, for all $i=1,\ldots ,4$. Consider the highest weight vectors
\[
f_{T_{\gen{\lambda_1}}}=(x_{1,1}^+)^n, \quad 
f_{T_{\gen{\lambda_2}}}=(x_{1,1}^+)^{n-1}x_{1,g}^+, \quad 
f_{T_{\gen{\lambda_3}}}=(x_{1,1}^+)^{n-2}(x_{1,g}^+)^2, \quad 
f_{T_{\gen{\lambda_4}}}=(x_{1,1}^+)^{n-2}x_{1,g^2}^+,
\]  
associated to the multipartitions
\[
\gen{\lambda_1}=((n)_{1^+}), \quad 
\gen{\lambda_2}=((n-1)_{1^+}, (1)_{g^+}), \quad 
\gen{\lambda_3}=((n-2)_{1^+}, (2)_{g^+}), \quad 
\gen{\lambda_4}=((n-1)_{1^+}, (1)_{g^2}^+).
\]  

Now, under the evaluation $x_{1,1}^+ \mapsto I_3, \,
x_{1,g}^+ \mapsto E_1, \, 
x_{1,g^2}^+ \mapsto E_1^2,$ we obtain $f_{T_{\gen{\lambda_1}}}, f_{T_{\gen{\lambda_2}}}, f_{T_{\gen{\lambda_3}}}, f_{T_{\gen{\lambda_4}}}\notin \textnormal{Id}^{\sharp}(C_{3}^\mathsf{g})$. By Proposition \ref{calmult}, we have $m_{
\langle \lambda_i \rangle}\geq 1$, for all $i=1,\ldots ,4$. Therefore, item {3)} follows.  

The remaining items can be proved in a similar way. 
	\end{proof}
	
	For $k\geq 2$, consider $E = \sum\limits_{i = 2}^{k-1} e_{i,i+1} + e_{2k-i,2k-i+1} \in UT_{2k}$ and define the following subalgebras of $UT_{2k}$
	\begin{align*}
		A_k &= \mbox{span}_F\{e_{11}+e_{2k,2k},E, \ldots , E^{k-2}; e_{12}, e_{13}, \ldots , e_{1k}, e_{k+1,2k}, e_{k+2,2k}, \ldots , e_{2k-1,2k}\},
		\\N_k &= \mbox{span}_F \{I_{2k}, E, \ldots , E^{k-2}; e_{12} - e_{2k-1,2k}, e_{13}, \ldots , e_{1k}, e_{k+1,2k}, e_{k+2,2k}, \ldots , e_{2k-2,2k}\},
		\\ U_k &= \mbox{span}_F \{ I_{2k}, E, \ldots , E^{k-2}; e_{12} + e_{2k-1,2k}, e_{13}, \ldots , e_{1k}, e_{k+1,2k}, e_{k+2,2k}, \ldots , e_{2k-2,2k}\}.
	\end{align*}
	
	\begin{definition}
		For $k\geq 2$, denote by $N_{k}^*, U_{k}^{*}$ and $A_{k}^*$ the algebras $N_k, U_k$ and $A_k$, respectively, with trivial $G$-grading and reflection involution. Furthermore, we define $N_k^\mathsf{g}, U_k^\mathsf{g}$ and $A_k^\mathsf{g}$ as the algebras $N_k, U_k$ and $A_k$, respectively, endowed with reflection involution and $G$-grading induced by $\mathsf{g}=(1,  g^{k-1}, 1^{k-1}, g)\in G^{2k}$, where $g\in G- \{1\}.$
	\end{definition}

We emphasize that the algebras above play an important role in the classification of subvarieties of $\textnormal{var}^{\sharp}(M_{g,\rho})$, as we can see below.
	
	\begin{theorem}\cite[Theorem $5.15$]{Mara} \label{31} 
		Let $g\in G$. If $A \in \textnormal{var}^{\sharp}(M_{g, \rho})$ then $A$ is $T_{G}^*$-equivalent to one of the following $(G,*)$-algebras: $M_{g, \rho}$, $B\oplus N$, $C\oplus N$ or $N$, where $B$ is a finite direct sum of distinct $(G,*)$-algebras in the set $\{N_{p}^{{g}}, U_{q}^{{g}}, A_{r}^{{g}}  \}$, $N$ is a nilpotent $(G,*)$-algebra and $C$ is a commutative $(G,*)$-algebra with trivial $G$-grading and trivial involution. 
	\end{theorem}

According to \cite{MN}, we have the following lemma. 
	
	\begin{lemma} \label{51} For $k\geq 2$ and $l>2$, we have $l_{n}^{\sharp}(A_{k}^*)=3k^2-5k+3,$ $l_{n}^{\sharp}(U_{l}^*)=\frac{3l^2-9l+8}{2}$ and $l_{n}^{\sharp}(N_{l}^*)=\frac{3l^2-11l+14}{2}.$ In particular, notice that $l_n^{\sharp}(U_3^*)=l_n^{\sharp}(N_3^*)=4$ and $l_{n}^{\sharp}(A_{2}^*)=5.$ 
	\end{lemma}


 By \cite[Theorems 5.12-5.14]{Mara}, we have the remark below.
	
			

   

	\begin{remark}
		For $k=2$ and $g\in G- \{1\}$ we may check that $U_{2}^*\sim_{T_G^*}C$, $N_{2}^\mathsf{g}\sim_{T_G^*} C_{2,*}^\mathsf{g}$ and $U_{2}^\mathsf{g}\sim_{T_G^*} C_2^\mathsf{g}$, where $C$ denotes a commutative algebra with trivial $G$-grading and trivial involution. Moreover, by \cite[Lemma 2]{MF} we have $N_{2}^*\sim_{T_G^*} C_{2,*}.$ Also, observe that $N_i^\mathsf{g} \in \textnormal{var}^{\sharp}(N_{i+1}^\mathsf{g})$, $U_i^\mathsf{g} \in \textnormal{var}^{\sharp}(U_{i+1}^\mathsf{g})$ and $A_i^\mathsf{g} \in \textnormal{var}^{\sharp}(A_{i+1}^\mathsf{g})$, for all $i\geq 2$ and $g\in G-\{1\}.$
	\end{remark}

	\begin{lemma} \label{52} For $g\in G-\{1\}$ we have,

    \begin{enumerate}

    \item[1)] $\textnormal{Id}^{\sharp}(N_3^\mathsf{g})=\langle x_{1,1}^-, x_{1,g}x_{2,g}, [x_{1,g}^+,{x_{1,1}^+}], x_{1,r} \rangle_{T_G^*}$ and $c_n^{\sharp}(N_{3}^{\mathsf{g}})=1+ 2n+ 2\displaystyle\binom{n}{2}$.

 \item[2)] $\mathcal{X}_n(N_{3}^{\mathsf{g}})=(\chi_{((n)_{1^+})}; \chi_{((n-1)_{1^+}, (1)_{g^+})};2\chi_{((n-1)_{1^+}, (1)_{g^-})}+\chi_{((n-2,1)_{1^+} , (1)_{g^-})}).$

        \item[3)] $\mathcal{X}_n(A_{2}^{\mathsf{g}})=(\chi_{((n)_{1^+})};2 \chi_{((n-1)_{1^+}, (1)_{g^+})};2\chi_{((n-1)_{1^+},  (1)_{g^-})}).$

        \item[4)] $l_{n}^{\sharp}(A_{2}^{\mathsf{g}})=l_{n}^{\sharp}(N_{3}^{\mathsf{g}})=5.$
    \end{enumerate}

	\end{lemma}
	
	\begin{proof}
The first item follows directly from Theorem~5.12 in \cite{Mara}. To prove item $2)$, by (\ref{293}) and (\ref{77}) we have
\[
c_n^{\sharp}(N_{3}^{\mathsf{g}})
= d_{(n)} 
+ \binom{n}{n-1} d_{(n-1)}d_{(1)} 
+ 2\binom{n}{n-1}d_{(n-1)}d_{(1)} 
+ \binom{n}{n-1}d_{(n-2,1)}d_{(1)}.
\]
Hence, it suffices to show that
\[
m_{\gen{\lambda_1}} \geq 1, 
\quad m_{\gen{\lambda_2}} \geq 1, 
\quad m_{\gen{\lambda_3}} \geq 2 
\quad \text{ and }\quad  m_{\gen{\lambda_4}} \geq 1,\, \mbox{ where}
\] \[
\gen{\lambda_1} = ((n)_{1^+}), \quad
\gen{\lambda_2} = ((n-1)_{1^+}, (1)_{g^+}), \quad
\gen{\lambda_3} = ((n-1)_{1^+}, (1)_{g^-}), \quad
\gen{\lambda_4} = ((n-2,1)_{1^+}, (1)_{g^-}).
\] For $\gen{\lambda_1}$ and $\gen{\lambda_2}$, consider the respective highest weight vectors
\[
f_1 = {(x_{1,1}^+)}^n 
\qquad \mbox{ and } \qquad
f_2 = {(x_{1,1}^+)}^{n-1}x_{1,g}^+.
\] By evaluating $x_{1,1}^+ \mapsto I_6$ and $x_{1,g}^+ \mapsto e_{13}+e_{46}$ we show that 
$f_1$ and $f_2$ are not identities of $N_3^\mathsf{g}$. 
Thus, by Proposition \ref{calmult} we have $m_{\gen{\lambda_1}}, m_{\gen{\lambda_2}} \geq 1$. For $\gen{\lambda_4}$, consider the multitableau
\[
\left(
\begin{array}{ll}
\begin{array}{|c|c|c|c|}
\hline
1 & 3 & \cdots & n-1 \\
\hline
\end{array}\,\,_{1^+} \\
\begin{array}{|c|}
\hline
n \\
\hline
\end{array}
\end{array},\,
\begin{array}{|c|}
\hline
2 \\
\hline
\end{array}\,\,_{g^-}
\right)
\]
whose highest weight vector is
\[
f_{T_{\gen{\lambda_4}}} 
= x_{1,1}^+x_{1,g}^- {(x_{1,1}^+)}^{n-3}x_{2,1}^+
- x_{2,1}^+x_{1,g}^- {(x_{1,1}^+)}^{n-2}.
\]
Evaluating $x_{1,1}^+ \mapsto I_6$, $x_{2,1}^+ \mapsto e_{23}+e_{45}$, and $x_{1,g}^- \mapsto e_{12}-e_{56}$ we obtain the element
$-e_{13}-e_{46} \neq 0$. 
Thus, $f_{T_{\gen{\lambda_4}}}\notin \textnormal{Id}(N_{3}^{\mathsf{g}})$ and then $m_{\gen{\lambda_4}} \geq 1$.

Finally, for the multipartition $\gen{\lambda_3}$ we consider the multitableaux
\[
\left( \, 
\begin{array}{|c|c|c|}
\hline
1 & \cdots & n-1 \\
\hline
\end{array}\,\,_{1^+},\,
\begin{array}{|c|}
\hline
n \\
\hline
\end{array}\,\,_{g^-}
\right)\mbox{ and }
\left(\, 
\begin{array}{|c|c|c|}
\hline
2 & \cdots & n-1 \\
\hline
\end{array}\,\,_{1^+},\,
\begin{array}{|c|}
\hline
1 \\
\hline
\end{array}\,\,_{g^-}
\right).
\]
Their corresponding highest weight vectors are
\[ f_4=
(x_{1,1}^+)^{n-1}x_{1,g}^- 
\quad \text{and} \quad f_5=
x_{1,g}^-(x_{1,1}^+)^{n-1}.
\]
Clearly these vectors are not identities of $N_3^\mathsf{g}$. 
To prove linear independence modulo $\textnormal{Id}^{\sharp}(N_3^\mathsf{g})$, consider
\[
\alpha (x_{1,1}^+)^{n-1}x_{1,g}^- 
+ \beta x_{1,g}^-(x_{1,1}^+)^{n-1} \equiv 0 
\quad \text{on } \textnormal{Id}^{\sharp}(N_3^\mathsf{g}).
\]
By evaluating $x_{1,1}^+ \mapsto I_6+e_{23}+e_{45}$ and $x_{1,g}^- \mapsto e_{12}-e_{56}$ gives $\alpha=\beta=0$. 
Hence $m_{\gen{\lambda_3}} \geq 2$.

The proof is now complete.
\end{proof}
	
	

	We consider the following subalgebra of $UT_4$
    $$W= F(e_{11}+\cdots + e_{44})+F(e_{12}+e_{34})+F(e_{13}+e_{24})+Fe_{14}.$$
    
	For this algebra, we consider the trivial involution $\nu_1$ and the involutions $\nu_2$ and $\nu_3$ defined below
	
	$$ \begin{pmatrix}
		a& b & c & d \\
		0&a & 0 & c\\
		0&0 & a & b \\
		0&0 & 0 & a 
	\end{pmatrix}^{\nu_2}= \begin{pmatrix}
		a& -b & -c & d \\
		0&a & 0 & -c\\
		0&0 & a & -b \\
		0&0 & 0 & a 
	\end{pmatrix} \mbox{ and } \begin{pmatrix}
		a& b & c & d \\
		0&a & 0 & c\\
		0&0 & a & b \\
		0&0 & 0 & a 
	\end{pmatrix}^{\nu_3}= \begin{pmatrix}
		a& -b & c & -d \\
		0&a & 0 & c\\
		0&0 & a & -b \\
		0&0 & 0 & a 
	\end{pmatrix}.$$

\begin{definition}
    For $g,h\in G$, we denote by $W_*^{g,h}$ as the algebra $W$ with involution $*\in \{\nu_1, \nu_2,\nu_3\}$ and $G$-grading determined by 
$$e_{11}+\cdots +e_{44}\in (W_*^{g,h})_1, \,e_{12}+e_{34}\in (W_*^{g,h})_g,\,e_{13}+e_{24}\in (W_*^{g,h})_h\mbox{ and } e_{14}\in (W_*^{g,h})_{gh}. $$  
\end{definition}

	\begin{lemma} Consider $g,h$ distinct elements in $G- \{1\}$ such that $gh\neq 1$. Then,
		\begin{enumerate}

			\item[1)] $\textnormal{Id}^{\sharp}(W_{\nu_1}^{g,h})=\langle x_{1,1}^-, x_{1,s}^-, x_{1,s}^+x_{2,s}^+,x_{1,g}^+x_{2,gh}^+, x_{1,h}^+x_{2,gh}^+, x_{1,r}\mid s\in \{g,h,gh\}\rangle_{T_G^*}. $
			
			\item[2)] $\textnormal{Id}^{\sharp}(W_{\nu_2}^{g,h})=\langle  x_{1,1}^-, x_{1,g}^+, x_{1,h}^+, x_{1,gh}^-, x_{1,g}^-x_{2,g}^-, x_{1,h}^-x_{2,h}^-, x_{1,gh}^+x_{2,gh}^+, x_{1,g}^-x_{2,gh}^+, x_{1,h}^-x_{2,gh}^+, x_{1,r} \rangle_{T_G^*}.$
	
			\item[3)] $\textnormal{Id}^{\sharp}(W_{\nu_3}^{g,h})=\langle x_{1,1}^-, x_{1,g}^+, x_{1,h}^-, x_{1,gh}^+, x_{1,g}^-x_{2,g}^-, x_{1,h}^+x_{2,h}^+, x_{1,gh}^+x_{2,gh}^+, x_{1,g}^-x_{2,gh}^-, x_{1,h}^+x_{2,gh}^-, x_{1,r}\rangle_{T_G^*}.$

			\item[4)] 
  $\mathcal{X}_n(W_{\nu_1}^{g,h})=(\chi_{((n)_{1^+})}; \chi_{((n-1)_{1^+}, (1)_{g^+})};\chi_{((n-1)_{1^+}, (1)_{h^+})};\chi_{((n-1)_{1^+}  ,  (1)_{gh^+} )};\chi_{((n-2)_{1^+}, (1)_{g^+}, (1)_{h^+})}).$
			
			\item[5)] $\mathcal{X}_n(W_{\nu_2}^{g,h})=(\chi_{((n)_{1^+})}; \chi_{((n-1)_{1^+}, (1)_{g^-})};\chi_{((n-1)_{1^+} , (1)_{h^-})};\chi_{((n-1)_{1^+},   (1)_{gh^+} )};\chi_{((n-2)_{1^+}, (1)_{g^-}, (1)_{h^-})}.$
			
			\item[6)] $\mathcal{X}_n(W_{\nu_3}^{g,h})=(\chi_{((n)_{1^+})}; \chi_{((n-1)_{1^+}, (1)_{g^-})};\chi_{((n-1)_{1^+}, (1)_{h^+})};\chi_{((n-1)_{1^+}, (1)_{gh^-} )};\chi_{((n-2)_{1^+},  (1)_{g^-}, (1)_{h^+})}).$
			
			\item[7)]    $c_n^{\sharp}(W_{\nu_i}^{g,h})=1+3n+
            {2}\displaystyle\binom{n}{2}$ and $l_{n}^{\sharp}(W_{\nu_i}^{g,h}) =5, \;i=1,2,3.$

		\end{enumerate}
	\end{lemma}

\begin{proof}
   
Consider the $T_G^*$-ideal $I$ generated by the identities in $1)$. Clearly $I \subseteq \textnormal{Id}^\sharp(W_{\nu_1}^{g,h})$.  
Let $f \in P_n^\sharp \cap \textnormal{Id}^\sharp(W_{\nu_1}^{g,h})$. Since the commutators $[x_{1,1}^+,x_{2,s}^+]$ and $[x_{1,g}^+,x_{2,h}^+]$ follow from the identities $x_{1,s}^-$, $s\in\{1,g,h,gh\}$, it follows that $P_n^\sharp$ modulo $P_n^\sharp \cap \textnormal{Id}^\sharp(W_{\nu_1}^{g,h})$ is generated by the monomials
\[
   x_{1,1}^+\cdots x_{n,1}^+, \quad
   x_{1,1}^+\cdots \widehat{x_{i,1}^+}\cdots x_{n,1}^+x_{i,g}^+, \quad
   x_{1,1}^+\cdots \widehat{x_{i,1}^+}\cdots x_{n,1}^+x_{i,h}^+,
\]
\[
   x_{1,1}^+\cdots \widehat{x_{i,1}^+}\cdots x_{n,1}^+x_{i,gh}^+, \quad
   x_{1,1}^+\cdots \widehat{x_{i,1}^+}\cdots \widehat{x_{j,1}^+}\cdots x_{n,1}^+x_{i,g}^+x_{j,h}^+.
\]
Hence, modulo $\textnormal{Id}^\sharp(W_{\nu_1}^{g,h})$, $f$ is a linear combination of these monomials. By multihomogeneity, we may assume that $f$ is one of the following polynomials
$$ \alpha x_{1,1}^+\cdots x_{n,1}^+,\;
   \beta x_{1,1}^+\cdots \widehat{x_{i,1}^+}\cdots x_{n,1}^+x_{i,g}^+,\;
   \gamma x_{1,1}^+\cdots \widehat{x_{i,1}^+}\cdots x_{n,1}^+x_{j,h}^+,
$$ $$
   \delta x_{1,1}^+\cdots \widehat{x_{i,1}^+}\cdots x_{n,1}^+x_{l,gh}^+\;\mbox{ or }\;
   \sigma x_{1,1}^+\cdots \widehat{x_{i,1}^+}\cdots \widehat{x_{j,1}^+}\cdots x_{n,1}^+x_{i,g}^+x_{j,h}^+.$$ Considering the evaluation $x_{m,1}^+ \mapsto I_4$, for all $ m$, $x_{i,g}^+ \mapsto e_{12}+e_{34}$, $
x_{j,h}^+ \mapsto e_{13}+e_{24}$ and $x_{l,gh}^+ \mapsto e_{14},$ we obtain $\alpha=\beta=\gamma=\delta=\sigma=0$. Therefore, $I=\textnormal{Id}^\sharp(W_{\nu_1}^{g,h})$. Moreover, the previous argument shows that there is no nontrivial linear combination of the above monomials giving an identity of $W_{\nu_1}^{g,h}$, and so they must constitute a basis to $P_n^\sharp$ modulo $P_n^\sharp\cap \textnormal{Id}^\sharp(W_{\nu_1}^{g,h})$. Therefore, 
\[
   c_n^\sharp(W_{\nu_1}^{g,h}) = 1+3n+2\binom{n}{2}.
\]

To prove the decomposition of the $\langle n \rangle$-cocharacters, we consider the highest weight vectors
\[
   (x_{1,1}^+)^{n-1},\quad (x_{1,1}^+)^{n-1}x_{1,g}^+,\quad
   (x_{1,1}^+)^{n-1}x_{1,h}^+,\quad (x_{1,1}^+)^{n-1}x_{1,gh}^+,\quad
   (x_{1,1}^+)^{n-2}x_{1,g}^+x_{1,h}^+,
\]
associated to the multipartitions
\[
   ((n)_{1^+}),\quad ((n-1)_{1^+},(1)_{g^+}),\quad ((n-1)_{1^+},(1)_{h^+}),\quad
   ((n-1)_{1^+},(1)_{gh^+}),\quad ((n-2)_{1^+},(1)_{g^+},(1)_{h^+}),
\]
respectively. These vectors are not identities of $W_{\nu_1}^{g,h}$, hence the corresponding multiplicities are nonzero.  
Item $4)$ now follows, since the codimension coincides with the value obtained above.

The other items are proved similarly.
\end{proof}

The proofs of the next lemmas are similar to the previous one and will be omitted.
 
	\begin{lemma} Consider $g\in G- \{1\}$ such that $|g|> 2$. Then,
		\begin{enumerate}
\item[1)]$\textnormal{Id}^{\sharp}(W_{\nu_3}^{g,g})=\langle [{x_{1,1}^+},x_{2,s}], x_{1,1}^-, x_{1,g^2}^+, x_{1,g}^-x_{2,g^2}^-,x_{1,g^2}^-x_{2,g^2}^-, x_{1,r}\mid s\in \{1,g,g^2\}\rangle_{T_G^*}.$

			\item[2)] $\textnormal{Id}^{\sharp}(W_{\nu_1}^{g,g^{-1}})=\langle x_{1,1}^-,x_{1,g}^-,x_{1,g^{-1}}^-, x_{1,r} \rangle_{T_G^*}.$		
			\item[3)] $\textnormal{Id}^{\sharp}(W_{\nu_2}^{g,g^{-1}})=\langle  x_{1,1}^-, x_{1,g}^+, x_{1,g^{-1}}^+, x_{1,r}\rangle_{T_G^*}.$
			
			\item[4)] $\textnormal{Id}^{\sharp}(W_{\nu_3}^{g,g^{-1}})=\langle  x_{1,g}^+, x_{1,g^{-1}}^-, [{x_{1,1}^+},x_{2,1}^+],[{x_{1,1}^+},x_{2,1}^+], [x_{1,g}^-,x_{2,g^{-1}}^+], x_{1,r} \rangle_{T_G^*}. $

	\item[5)] $\mathcal{X}_n(W_{\nu_3}^{g,g})=(\chi_{((n)_{1^+})};\, \chi_{((n-1)_{1^+}, (1)_{g^+})};\,\chi_{((n-1)_{1^+}, (1)_{g^-})};\,\chi_{((n-1)_{1^+} ,  (1)_{{g^2}^-} )};\,\chi_{((n-2)_{1^+} , (1)_{g^+}, (1)_{g^-})})$.

   \item[6)] $\mathcal{X}_n(W_{\nu_1}^{g,g^{-1}})=(\chi_{((n)_{1^+})}; \chi_{((n-1)_{1^+},  (1)_{g^+})};\chi_{((n-1)_{1^+},  (1)_{(g^{-1})^+})};\chi_{((n-2)_{1^+},(1)_{g^+} , (1)_{(g^{-1})^+})}).$
			
			\item[7)] $\mathcal{X}_n(W_{\nu_2}^{g,g^{-1}})=(\chi_{((n)_{1^+})}; \chi_{((n-1)_{1^+}, (1)_{g^-})};\chi_{((n-1)_{1^+} , (1)_{(g^{-1})^-})};\chi_{((n-2)_{1^+}, (1)_{g^-} , (1)_{(g^{-1})^-})}.$
			
			\item[8)] $\mathcal{X}_n(W_{\nu_3}^{g,g^{-1}})=(\chi_{((n)_{1^+})};{\chi_{((n-1)_{1^+}, (1)_{1^-})}}; \chi_{((n-1)_{1^+}, (1)_{g^-})};\chi_{((n-1)_{1^+}, (1)_{(g^{-1})^+})};\chi_{((n-2)_{1^+},  (1)_{g^-}, (1)_{(g^{-1})^+})}).$
			
			\item[9)] $c_n^{\sharp}(W_{\nu_i}^{g,g^{-1}})=1+2n+\displaystyle {2}\binom{n}{2}$, $c_n^{\sharp}(W_{\nu_3}^{g,g^{-1}})=c_n^{\sharp}(W_{\nu_3}^{g,g})=1+3n+{2}\displaystyle\binom{n}{2}$, $l_{n}^{\sharp}(W_{\nu_i}^{g,g^{-1}})=4$\;and\;$l_{n}^{\sharp}(W_{\nu_3}^{g,g^{-1}})=l_{n}^{\sharp}(W_{\nu_3}^{g,g})=5$, $i=1,2$.
			
		\end{enumerate}
	\end{lemma}

\begin{lemma}
Let $g,h\in G-\{1\}$ with $|h|=2$. Then,

		\begin{enumerate}

 \item[1)] $\textnormal{Id}^{\sharp}(W_{\nu_2}^{1,g})=\langle [{x_{1,1}^+},x_{2,1}], [{x_{1,1}^+},x_{2,g}], x_{1,1}^-x_{2,1}^-,x_{1,1}^-x_{2,g}^+,[x_{1,1}^-,x_{2,g}^-], x_{1,g}^+x_{2,g}^+, x_{1,g}^+x_{2,g}^-,x_{1,g}^-x_{2,g}^-, x_{1,r}\rangle_{T_G^*}$.
       
   \item[2)] $\textnormal{Id}^{\sharp}(W_{\nu_3}^{1,g})=\langle [{x_{1,1}^+},x_{2,1}], [{x_{1,1}^+},x_{2,g}], x_{1,1}^-x_{2,1}^-,x_{1,1}^-x_{2,g}^-,[x_{1,1}^-,x_{2,g}^+], x_{1,g}^-x_{2,g}^-, x_{1,g}^-x_{2,g}^+,x_{1,g}^+x_{2,g}^+, x_{1,r} \rangle_{T_G^*}$.

			\item[3)]  $\textnormal{Id}^{\sharp}(W_{\nu_3}^{h,h})=\langle [{x_{1,1}^+},x_{2,s}], x_{1,1}^-x_{2,1}^-, x_{1,1}^-x_{2,h},x_{1,h}^+x_{2,h}^+, x_{1,h}^-x_{2,h}^-,[x_{1,h}^+,x_{2,h}^-]
			,x_{1,r}\mid s\in \{1,h\}\rangle_{T_G^*} $.

  \item[4)] $\mathcal{X}_n(W_{\nu_2}^{1,g})=(\chi_{((n)_{1^+})}; \chi_{((n-1)_{1^+}, (1)_{1^-})};\chi_{((n-1)_{1^+}, (1)_{g^+})};\chi_{((n-1)_{1^+}, (1)_{g^-} )};\chi_{((n-2)_{1^+},  (1)_{1^-}, (1)_{g^-})})$.

   \item[5)] $\mathcal{X}_n(W_{\nu_3}^{1,g})=(\chi_{((n)_{1^+})}; \chi_{((n-1)_{1^+}, (1)_{1^-})};\chi_{((n-1)_{1^+}, (1)_{g^+})};\chi_{((n-1)_{1^+}, (1)_{g^-} )};\chi_{((n-2)_{1^+},  (1)_{1^-}, (1)_{g^+})})$.

			\item[6)]  $\mathcal{X}_n(W_{\nu_3}^{h,h})=(\chi_{((n)_{1^+})};\chi_{((n-1)_{1^+}, (1)_{1^-})}; \chi_{((n-1)_{1^+}, (1)_{h^+})};\chi_{((n-1)_{1^+}, (1)_{h^-})};\chi_{((n-2)_{1^+}, (1)_{h^+}, (1)_{h^-})}).$ 
   
   	\item[7)] $c_n^{\sharp}(A)=1+3n+{2}\displaystyle\binom{n}{2}$ \;and\; $l_n^{\sharp}(A)=5,$ for all $A\in \{W_{\nu_2}^{1,g},W_{\nu_3}^{1,g},W_{\nu_3}^{h,h} \}.$
   
		\end{enumerate}
	\end{lemma}

\section{\texorpdfstring{$(G,*)$-colength bounded by 1}{(G,*)-colength bounded by 1}}

This section aims to characterize the varieties of $(G,*)$-algebras whose corresponding colength sequences are bounded by $1$.

It is known that if $A = F + J(A)$ is a finite-dimensional $(G,*)$-algebra over $F$, where $J(A)$ is its Jacobson radical, then $J = J(A)$ is a $T_G^*$-ideal of $A$ (see \cite[Theorem 2.8]{Lorena}). Moreover, $J$ can be decomposed as the direct sum of $G$-graded subspaces
	\begin{equation} \label{6}
		J=J_{00}+J_{10}+J_{01}+J_{11},
	\end{equation} where     $$J_{00}=\{j\in J\mid 1_Fj=j1_F=0\},\quad J_{10}=\{j\in J\mid 1_Fj=j 
    \mbox{ and }j1_F=0\}$$
      $$J_{01}=\{j\in J\mid 1_Fj=0 
    \mbox{ and }j1_F=j\}\quad \mbox{ and }\quad J_{11}=\{j\in J\mid 1_Fj=j1_F=j\},$$ (see, for instance, \cite[Lemma 2]{Giambruno}). By definition, it is clear that the subspaces $J_{00}$ and $J_{11}$ are invariant under the involution while $(J_{10})^{*} = J_{01}$. Furthermore, for $i, k, r, s \in \{0,1\}$ we have $J_{ik} J_{rs} \subseteq \delta_{kr} J_{is}$, where $\delta_{kr}$ denotes the Kronecker delta function. Therefore, we may further decompose $J_{00}$ and $J_{11}$ into symmetric and skew homogeneous components as follows:
	$$J_{00}= \underset{g\in G}{\sum} ((J_{00})_g^{+}+ (J_{00})_g^{-})\mbox{ and }J_{11}= \underset{g\in G}{\sum} ((J_{11})_g^{+}+ (J_{11})_g^{-}),$$ where $(J_{ii})_g^+=\{a\in (J_{ii})_g\mid a^*=a\}$ and $(J_{ii})_g^-=\{a\in (J_{ii})_g\mid a^*=-a\}$, $g\in G$ and $i\in \{0,1\}$.

  Motivated by Theorem \ref{1}, we investigate $(G,*)$-algebras of type $A = F + J(A)$ and establish conditions in order to reduce the components in the decomposition given in (\ref{6}). To reach our goal, we consider the set $$\mathcal{I}_3=\{A_2^*, A_2^{\mathsf{g}}\mid \mathsf{g}=(1,g,1,g)\in G^4\}.$$
	
	\begin{lemma} \label{3}
		Let $A=F+J(A)$ be a $(G,*)$-algebra such that $Q\notin \textnormal{var}^{\sharp}(A)$, for all $Q \in \mathcal{I}_3$. Then, $J_{10}=J_{01}=\{0\}$.
	\end{lemma}
	
	\begin{proof}
		Since $F+(J(A))_{1}$ is a $(G,*)$-subalgebra of $A$ with trivial $G$-grading, by Lemma $14$ in \cite{LM} we have $(J_{10})_1=(J_{01})_1=0$. Consider $g\in G- \{1\}$ and suppose that there exists a nonzero element $a\in (J_{10})_{g}$. Let $R$ be the \Gstar subalgebra of $A$ generated by $1_F,a$ and $a^*$ endowed with induced $G$-grading and induced involution. Let $I$ be the $T_G^*$-ideal generated by $aa^*$ and $a^*a$. Hence, $R/I$ is a $(G,*)$-algebra with $$(R/I)_1^{+}=\mbox{span}_F\{\overline{1_F}\},\,\, (R/I)_g^{+}=\mbox{span}_F\{\overline{a}+\overline{a^*}\},\,\,(R/I)_g^{-}=\mbox{span}_F\{\overline{a}-\overline{a^*}\}.$$
		
		Define the isomorphism of $(G,*)$-algebras $\varphi: R/I \rightarrow A_2^\mathsf{g}$ given by $\varphi(\overline{1_F})= e_{11}+e_{44},\varphi( \overline{a})=e_{12}$ and $\varphi(\overline{a^*})= e_{34}.$ Hence, $A_2^\mathsf{g}\in \textnormal{var}^{\sharp}(A)$.
	\end{proof}
	
	As a consequence, under the hypothesis of the previous lemma we have $A= (F+J_{11})\oplus J_{00}$. Since $J_{00}$ is a nilpotent $(G,*)$-subalgebra of $A$, in the following results of this section, unless otherwise mentioned, we assume that $A$ is a \Gstar algebra of type $F+J_{11}$.

\begin{lemma} \label{73} Let $g\in G-\{1\}$.
\begin{enumerate}
\item[1)] If $C_{k,*}\notin \textnormal{var}^{\sharp}(A)$ then $b^{k-1}=0$, for all $b\in (J_{11})_1^- $.

    \item[2)] If $C_{k}^\mathsf{g}\notin \textnormal{var}^{\sharp}(A)$ then $b^{k-1}=0$, for all $b\in (J_{11})_g^+ $.

    \item[3)] If $ C_{k,*}^\mathsf{g}\notin \textnormal{var}^{\sharp}(A)$ then $b^{k-1}=0$, for all $b\in (J_{11})_g^-$.
\end{enumerate}
	\end{lemma}
	
	\begin{proof} The first item follows by Lemma $4.6$ in \cite{MN}. Now, suppose that there exists $b\in (J_{11})_g^+\cup (J_{11})_g^-$ such that $b^{k-1}\neq 0$ and let $R$ be the $(G,*)$-subalgebra of $A$ generated by $1_F$ and $b$. Considering $I$ the $T_G^*$-ideal of $R$ generated by $b^{k}$, we have that $R/I$ is a $(G,*)$-algebra linearly generated by $\overline{1_F}, \overline{b},  \ldots , \overline{b}^{k-1}$. Note that the map $\varphi: R/I \rightarrow C_{k}$ given by $\varphi(\overline{1_F})=e_{11}+\cdots +e_{kk}$ and $\varphi(\overline{b})=e_{12}+\cdots+ e_{k-1\, k}$ defines an isomorphism of $(G,*)$-algebras and we have
		$R/I\cong C_k^\mathsf{g}$ if $b\in (J_{11})_g^+$, while
		 $R/I\cong C_{k,*}^\mathsf{g}$ if $b\in (J_{11})_g^-$.
\end{proof}

The $T_G^*$-ideals of the $(G,*)$-algebras $(FC_2)_{*}$, $(FC_2)^{\sharp}$, $FC_p$ and $M_{g,\rho}$ were completely described in \cite{Mara,Valenti}. As a consequence, we have the following. 
    
    \begin{remark} \label{12}
		For $k \geq 2$ we have
		\begin{enumerate}
						\item[1)] $C_{k,*}\in \textnormal{var}^{\sharp}((FC_2)_{*}).$
			\item[2)] If $|G|$ is even and $h\in G$ with $|h|=2$ and $C_2=\langle h\rangle$ then  $C_{k,*}^\mathsf{h}\in \textnormal{var}^{\sharp}((FC_2)^{\sharp}).$
			\item[3)] For all prime $p\mid  |G|$, $h\in G$ with $|h|=p$ and $C_p=\langle h\rangle$ we have $C_{k}^\mathsf{h}\in \textnormal{var}^{\sharp}(FC_p)$.

   	\item[4)] $N_{3}^*$, $U_{3}^*$, $A_{2}^* \in \textnormal{var}^{\sharp}(M_{1,\rho})$ and $N_{3}^{\mathsf{g}}$, $U_{3}^{\mathsf{g}}$, $A_{2}^{\mathsf{g}} \in \textnormal{var}^{\sharp}(M_{g,\rho}),$ $g\in G- \{1\}$.
		\end{enumerate}
	\end{remark}

The next theorem shows that the codimension and the colength sequence are closely related in the case where $l_n^\sharp(A)$ is bounded by $1$. We also present an alternative characterization of these varieties via the exclusion of algebras from the variety.

\begin{theorem} Let $A$ be a finite dimensional $(G,*)$-algebra over a fiel of characteristic zero. The following conditions are equivalent
\begin{enumerate}
    \item[1)] $A_{2}^*,A_2^{\mathsf{g}},C_{2,*}, C_{2}^\mathsf{g}, C_{2,*}^\mathsf{g}\notin \textnormal{var}^{\sharp}(A)$, for all $g\in G- \{1\}.$
    \item[2)] $A\sim_{T_{G}^*} C\oplus N$ or $A\sim_{T_{G}^*} N$, where $N$ is a nilpotent \Gstar algebra and $C$ is a commutative non-nilpotent algebra with trivial involution and trivial $G$-grading.
     \item[3)] $c_{n}^{\sharp}(A)\leq q,$ for some constant $q$.
     \item[4)] $c_n^{\sharp}(A)\leq 1$.
       \item[5)] $l_{n}^{\sharp}(A)\leq 1.$
\end{enumerate}
\end{theorem}

\begin{proof}
    It is trivial to notice that condition $2)$ implies both conditions $ 3)$, $4)$ and $5)$. Moreover, it is easy to check that both $3)$, $4)$ and $5)$ imply $1)$. Then, we concentrate our effort to prove that condition $1)$ implies condition $2)$. 
    
    Assume that $A_{2}^*,$ $A_2^\mathsf{g},$ $C_{2,*}, C_{2}^\mathsf{g},$ $C_{2,*}^\mathsf{g}$ $\notin \textnormal{var}^{\sharp}(A)$, for all $g\in G- \{1\}$. By Remark \ref{12} and Theorem \ref{1}, $A$ has polynomial growth. Therefore, $$A\sim_{T_{G}^*} B_1\oplus \cdots \oplus B_m,$$ where $B_i$ is a nilpotent \Gstar algebra or $B_i=F+J(B_i)$, for $1\leq i\leq m$. If $B_i$ is nilpotent for all $i$ then we are done. Suppose that there exists $i\in \{1,\ldots , m\}$ such that $B_i=F+J(B_i)$. Consider $J(B_i)=J_{00}+J_{10}+J_{01}+J_{11}$ the decomposition of $J(B_i)$ as given in (\ref{6}). By Lemmas \ref{3} and \ref{73}, we get $$J_{01}=J_{10}= (J_{11})_h=(J_{11})_1^- =\{0\},$$ for all $h\in G-\{1\}$. Consequently, we can rewrite  
 $B_i= (F+ (J_{11})_1^{+}) \oplus J_{00},$ where $J_{00}$ is a nilpotent $(G,*)$-algebra. Since $[(J_{11})_1^+,(J_{11})_1^+]\subseteq (J_{11})_1^-=\{0\}$ we conclude that $F+ (J_{11})_1^{+}$ is a commutative $(G,*)$-algebra endowed with trivial $G$-grading and trivial involution, and so the condition $2)$ is satisfied.
 \end{proof}

\section{\texorpdfstring{Classifying $(G,*)$-varieties of $(G,*)$-colength bounded by $4$}{Classifying (G,*)-varieties of (G,*)-colength bounded by 4}}

The main purpose of this section is to give the classification of $(G,*)$-varieties that have the corresponding colength sequence bounded by $4$.

We still assume that $A$ is a finite-dimensional $(G,*)$-algebra of the type $F+J_{11}$.

\begin{remark} \label{187}
   Suppose that there exist symmetric or skew elements $a\in (J_{11})_g$ and $b\in (J_{11})_h$ such that $ab\neq 0$, for some $g,h\in G$. If $ba\in Fab$ then there exists $\alpha \in F$ such that $ba=\alpha ab$. Since $ab = \pm (ba)^*= \pm  (\alpha ab)^*= \alpha ba=\alpha^2 ab$ we have either $ab=ba$ or $ab=-ba$.
\end{remark}

  Consider the following set of $(G, \ast)$-algebras $$\mathcal{I}_4=\{C_{3}^\mathsf{g}, C_{3,*}^\mathsf{g},C_{4,*},C_{4}^\mathsf{q}, C_{4,*}^\mathsf{q},U_{3,*}, W_{\imath}^{r,s}, W_{\jmath}^{1,g}, \mathcal{G}_{2,\diamond }^{r,s},\mathcal{G}_{2,\eta }^{1,g},W_{\nu_3}^{g,g}, \mathcal{G}_{2,\diamond}^{g,g}\mid g,h,r,s,q\in G-\{1\},$$ $$  |g|>2, |q|=2, r\neq s, \imath\in \{\nu_1, \nu_2, \nu_3\},\jmath \in \{\nu_2,\nu_3\}, \diamond \in \{\psi,\tau,  \gamma\},\eta \in \{\tau, \gamma\}\}.$$ 
	
	\begin{lemma} \label{10}
		Let $A$ be a $(G,*)$-algebra such that $Q\notin \textnormal{var}^{\sharp}(A)$, for all $Q\in \mathcal{I}_4$. Then, $(J_{11})_g(J_{11})_h=(J_{11})_r(J_{11})_r=(J_{11})_1^- (J_{11})_r=[(J_{11})_1^+,(J_{11})_1^+]=\{0\}$, for all $g,h,r\in G- \{1\},$ with $g\neq h$ and $|r|>2$.
	\end{lemma}

\begin{proof}

Since $F+(J_{11})_1$ is a $(G,*)$-subalgebra of $A$ with trivial $G$-grading and $U_{3,*}\notin \textnormal{var}^{\sharp}(A)$, by \cite[Lemma 4.7]{MN} we have $[(J_{11})_1^+,(J_{11})_1^+]=\{0\}.$

	Suppose that there exist symmetric or skew elements $a\in (J_{11})_g$ and $b\in (J_{11})_h$ such that $ab\neq 0$, for some $g,h\in G- \{1\}$. Let $R$ be the $(G,*)$-subalgebra of $A$ generated by $1_F,a$ and $b$ with induced $G$-grading and induced involution.

First, we consider the case $g=h$ with $|g|>2$. Since $C_{3}^\mathsf{g}, C_{3,*}^\mathsf{g}\notin \textnormal{var}^{\sharp}(A),$ by the previous lemma we have $a^2=b^2=0$. If both $a$ and $b$ are symmetric or both are skew we also have $(a+b)^2=0$ and so $ab=-ba$. Therefore, the map $\varphi_1: R \rightarrow \mathcal{G}_2$ given by $\varphi_1({1_F})= 1, \varphi_1({a})= e_1, \varphi_1({b})=e_2$ and $\varphi_1({ab})=e_1e_2$ defines an isomorphism of $(G,*)$-algebras in the following cases 

\begin{enumerate}
			\item[1.] if $a\in (J_{11})_g^+$, $b\in (J_{11})_g^+$ then $R\cong \mathcal{G}_{2,\psi}^{g,g};$ 
			\item[2.] if $a\in (J_{11})_g^-$, $b\in (J_{11})_g^-$ then $R\cong \mathcal{G}_{2,\tau}^{g,g}.$
		\end{enumerate}

  Now, suppose that $a$ is skew and $b$ is symmetric. If $ba\in Fab$, by Remark \ref{187} either $ba=-ab$ or $ba=ab$. In the first case, the map $\varphi_1: R\rightarrow \mathcal{G}_2$ defined above is an isomorphism of $(G,*)$-algebras between $R$ and $\mathcal{G}_{2,\gamma}^{g,g}.$ In the second case, the map $\varphi_2: R \rightarrow W$ given by $\varphi_2({1_F})= e_{11}+\cdots + e_{44}, \varphi_2({b})=e_{12}+e_{34}, \varphi_2({a})=e_{13}+e_{24}$ and $\varphi_2({ab})=e_{14}$ defines an isomorphism of $(G,*)$-algebras between $R$ and $W_{\nu_3}^{g,g}$. 
  On the other hand, if $ba\notin Fab$ we consider $I=\langle ab+ba \rangle_{T_G^*}$ the $T_G^*$-ideal of $R$ generated by $ab+ba$. Since $C_{3,*}^\mathsf{g}\notin \textnormal{var}^{\sharp}(A)$ then $I$ is linearly generated by $ab+ba, aba, bab$ and $abab$, and so $a,b$ and $ab$ do not belong to $I$. Therefore, $R/I$ is a $(G,*)$-algebra linearly generated by $\overline{1_F},\overline{a}, \overline{b}, \overline{ab}$ satisfying $\overline{a}^2=\overline{b}^2=\overline{ab}+\overline{ba}=0$. Hence $R/I\cong \mathcal{G}_{2, \gamma}^{g,g}.$

Now we consider the case  $g\neq h$. Taking the quotient of $R$ by the $T_G^*$-ideal generated by $a^2$ and $b^2$, without loss of generality we can assume $a^2=b^2=0$ such that $ab\neq 0$. Also, if $ba\in Fab$, then by Remark \ref{187} either $ab=ba$ or $ab=-ba$. In the first case, the map $\varphi_2: R \rightarrow W$ defined above defines an isomorphism of $(G,*)$-algebras in the following cases 
		
		\begin{enumerate}
			\item[1.] if $a\in (J_{11})_g^+$ and $b\in (J_{11})_h^+$ then $R\cong W_{\nu_1}^{g,h};$ 
  
			\item[2.] if $a\in (J_{11})_g^-$ and $b\in (J_{11})_h^-$ then $R\cong W_{\nu_2}^{g,h};$

			\item[3.] if $a\in (J_{11})_g^+$ and $b\in (J_{11})_h^-$ then $R\cong W_{\nu_3}^{g,h}.$
		\end{enumerate}

In the second case, the map $\varphi_1: R \rightarrow \mathcal{G}_2$ defined above defines an isomorphism of $(G,*)$-algebras in the following cases

\begin{enumerate}
			\item[1.] if $a\in (J_{11})_g^+$ and $b\in (J_{11})_h^+$ then $R\cong \mathcal{G}_{2,\psi}^{g,h};$ 
 
			\item[2.] if $a\in (J_{11})_g^-$ and $b\in (J_{11})_h^-$  then $R\cong \mathcal{G}_{2,\tau}^{g,h};$
 
			\item[3.] if $a\in (J_{11})_g^-$ and $b\in (J_{11})_h^+$ then $R\cong \mathcal{G}_{2,\gamma}^{g,h};$
  
		\end{enumerate}

  Therefore, we may assume that $ba\notin Fab$. In this case, we consider $I$ the $T_G^*$-ideal of $R$ generated by $ab+ba$.  Suppose that $|gh|=1$, that is, $h=g^{-1}$. If $a$ is symmetric and $b$ is skew (resp. both $a$ and $b$ are either symmetric or skew) then, since $C_{4,*}\notin \textnormal{var}^{\sharp}(A)$ we have $(ab+ba)^3=0$ (resp. $(ab-ba)^3=0$) and so $a,b,ab\notin I$. In this case, $R/I \cong \mathcal{G}_{2,*}^{g,g^{-1}}$ and so $\mathcal{G}_{2,*}^{g,g^{-1}}\in \textnormal{var}^{\sharp}(A)$, for some $*\in \{\psi, \tau, \gamma\}.$ 
  
  Now suppose $|gh|>2$ (resp. $|gh|=2$). In this case, we have $C_3^{gh},C_{3,*}^{gh}\notin \textnormal{var}^{\sharp}(A)$ (resp. $C_4^{gh},C_{4,*}^{gh}\notin \textnormal{var}^{\sharp}(A))$ and so $(ab+ba)^2=0$ (resp. $(ab+ba)^3=0$). In both case we can check that $a,b,ab\notin I.$ Thus, $R/I$ is a $(G,*)$-algebra linearly generated by $\overline{1_F},\overline{a},\overline{b}$ and $\overline{ab}$ satisfying  $\overline{a}^2=\overline{b}^2=\overline{ab}+\overline{ba}=0$. In this situation we have $R/I\cong \mathcal{G}_{2,*}^{g,h}$ and so $ \mathcal{G}_{2,*}^{g,h}\in \textnormal{var}^{\sharp}(A)$, for some $*\in \{\psi, \tau, \gamma\}.$

  Finally, using the previous ideas we may check that  $(J_{11})_1^- (J_{11})_h=\{0\}$,  $h\in G$ with $|h|>2,$ since $C_3^h, C_{3,*}^{h},W_\jmath^{1,h}, \mathcal{G}_{2,\diamond}^{1,h}\notin \textnormal{var}^{\sharp}(A)$, for all $\jmath\in \{\nu_2, \nu_3\}$, $\diamond \in \{\tau, \gamma\}$.
 \end{proof}

	\begin{lemma} \label{306}
		Let $A=F+J_{11}$ be a $(G,*)$-algebra such that $Q\notin \textnormal{var}^{\sharp}(A)$, for all $Q\in \mathcal{I}_4$. Then, $A\sim_{T_G^*} B$ where 
\begin{equation}\label{DecompSpecial}B=  \left({\bigoplus_{|g|\leq  2} } (F+(J_{11})_1+(J_{11})_g)\right)\oplus \left({\bigoplus_{|h|>2} } (F +(J_{11})_1^+ +(J_{11})_h\right).
  \end{equation}
		
	\end{lemma}
	
	\begin{proof}
	First, notice that $F+(J_{11})_1+(J_{11})_g$ is a subalgebra of $A$, for all $g\in G$ with $|g|\leq 2$.
        
    By Lemma \ref{10}, the polynomials $x_{1,g}x_{2,q},\, x_{1,h}x_{2,h}$, $x_{1,1}^-x_{2,h}$ and $[{x_{1,1}^+},x_{2,1}^+]$ are $(G,*)$-identities of $A$, for all $g,q,h\in G- \{1\}$ with $g\neq q$, $|h|>2$ and $x_i\in \{x_i^+,x_i^-\}.$ Therefore, $F +(J_{11})_1^+ +(J_{11})_h$ is also a $(G,*)$-subalgebra of $A$, for all $h\in G$ with $|h|>2$. Hence, $\textnormal{Id}^{\sharp}(A)\subseteq \textnormal{Id}^{\sharp}(B)$. Now, consider $f\in P_n^{\sharp}\cap \textnormal{Id}^{\sharp}(B)$. Reducing $f$ modulo the $\textnormal{T}_G^*$-ideal
 $$\langle x_{1,g}x_{2,q},\, x_{1,h}x_{2,h}, x_{1,1}^-x_{2,h},[{x_{1,1}^+},x_{2,1}^+]  \mid g,q,h\in G- \{1\},g\neq q,|h|>2\rangle_{T_G^*} $$ and using the multihomogeneity of $\textnormal{T}_G^*$-ideals we may rewrite $f$ as a linear combination of polynomials of one of the following types
		\begin{enumerate}
			\item[1.] $x_{i_1,1}^+\cdots x_{i_{j-1},1}^+x_{i_j,h} x_{i_{j+1},1}^+\cdots x_{i_{n},1}^+$;
			\item[2.] $x_{\sigma(1),g_{i_1}}\cdots x_{\sigma(n),g_{i_n}},$ where $g_{i_t}\in \{1,g\}$;
		\end{enumerate} for some $g,h\in G$ with $|g|\leq 2$, $|h|>2$, $x_i\in\{x_i^+,x_i^-\}$, $\sigma \in S_n$ and $i_1< \cdots < i_{j-1}, i_{j+1}< \cdots < i_{n}. $ 
		
		In the first case, we have to evaluate $f$ on elements of $F + (J_{11})_1^+ + (J_{11})_{h}$. Since $f$ is a $(G,*)$-identity of $B$ we conclude that $f\equiv 0$ on $A$. 
		In the second case, we must evaluate $f$ on the subalgebra $F+(J_{11})_1+(J_{11})_g$ of $A$. Since $f$ is an identity of $F+(J_{11})_1+(J_{11})_g$ we have $f\equiv 0$ on $A$. Therefore, $\textnormal{Id}^\sharp({A})= \textnormal{Id}^\sharp({B})$.
	\end{proof}

The previous lemma will be fundamental in the proof of our main result. In fact, the decomposition (\ref{DecompSpecial}) seems to be much simpler and more valuable for our goal, since $F+(J_{11})_1+(J_{11})_g)$, $|g|\leq 2$, is a $(G,*)$-algebra with induced $\mathbb{Z}_2$-grading and $F +(J_{11})_1^+ +(J_{11})_r$, $|r|>2$, satisfies the identities $x_{1,r}x_{2,r}$ and $x_{1,1}^-.$

In the meantime, we assume that $|G|$ is even and for any element $g\in G$ with $|g|=2$ we consider the set of $(G,*)$-algebras below endowed with their respective induced $\mathbb{Z}_2$-grading $$\mathcal{I}_2=\{A_{2}^*,N_{3}^*, U_{3}^*,A_2^\mathsf{g} , N_3^\mathsf{g} , U_3^\mathsf{g}, C_{4,*}, C_{4}^\mathsf{g} , C_{4,*}^\mathsf{g}, C_{2,*}\oplus C_2^\mathsf{g}\oplus C_{2,*}^\mathsf{g}, \mathcal{G}_{3,\tau}^{1,1}, \mathcal{G}_{2,\imath}^{1,g},\mathcal{G}_{2,\diamond}^{g,g}, \mathcal{G}_{2,\tau}^{1,1}\oplus C_{3,*}$$ $$
		\mathcal{G}_{2,\tau}^{1,1} \oplus K_1,C_{3,*}\oplus K_1 ,C_{3,*}^\mathsf{g}\oplus K_2, C_3^\mathsf{g}\oplus K_3, W_{\nu_3}^{g,g}, W_{\jmath}^{1,g} \},$$ where $ \diamond \in \{\tau, \psi, \gamma\}, \imath\in \{\tau, \gamma\}, \jmath \in \{\nu_2,\nu_3\}, K_i\in \mathcal{D}^g,K_1\neq C_{2,*}, K_2\neq C_{2,*}^\mathsf{g}$ and $K_3\neq C_2^\mathsf{g}$.
	
We emphasize that the algebras in $\mathcal{I}_2$ were considered by do Nascimento and Vieira \cite[Theorem $8.8$]{NV} in the classification of $(\mathbb{Z}_2,*)$-varieties whose $n$th $(\mathbb{Z}_2,*)$-colengths are bounded by $3$. As a consequence, we have the following result.

\begin{theorem} \label{338}  Let $A$ be a finite-dimensional $(G,*)$-algebra over a field $F$ of characteristic zero with induced $\mathbb{Z}_2$-grading by an element $g\in G$. If $Q\notin \textnormal{var}^{\sharp}(A)$, for all $Q\in \mathcal{I}_2$, then $A$ is $T_{G}^*$-equivalent to one of the following $(\mathbb{Z}_2,*)$-algebras: $N, C\oplus N, \mathcal{G}_{2,\tau} \oplus N,  D_1\oplus N, D_1\oplus D_2\oplus N, D_3\oplus N ,$ where    $D_1,D_2\in \mathcal{D}^g$ with $D_1\neq D_2$, $D_3\in \{C_{3,*}, C_3^\mathsf{g}, C_{3,*}^\mathsf{g}\}$, $N$ is a nilpotent $*$-superalgebra and $C$ is a commutative non-nilpotent algebra with trivial involution and trivial grading.
	
	\end{theorem}

    For all $S_1\in \mathcal{D}- \{C_{2}^\mathsf{h}\},\, S_2\in \mathcal{D}- \{C_{2,*}^\mathsf{h}\},\, T\in \mathcal{D}-\{C_{2,*}\}$ and $ P\in \mathcal{D}_2$ we consider the sets
	$$\mathcal{I}_5=\{C_3^{\mathsf{h}}\oplus S_1, C_{3,*}^\mathsf{h}\oplus S_2, C_{3,*}\oplus T, C_{3,*}\oplus \mathcal{G}_{2,\tau}^{1,1}, \mathcal{G}_{2,\tau}^{1,1} \oplus T, P \mid |h|=2 \}\mbox{ and }\mathcal{I}_6=\{N_3^\mathsf{h},U_3^\mathsf{h}\mid |h|>2\}.$$ Define 
\begin{equation}\label{ConjuntoI}
\mathcal{I}=\displaystyle \bigcup_{k=2}^6 \mathcal{I}_k.
\end{equation}
	
	\begin{theorem} \label{157} Let $G$ be a finite abelian group of even order and $A$ be a finite-dimensional $(G,*)$-algebra. The following conditions are equivalent
		\begin{enumerate}
			\item[1)] $Q\notin \textnormal{var}^{\sharp}(A)$, for all $Q\in \mathcal{I}.$
			\item[2)] $A$ is $T_G^*$-equivalent to one of the following $(G,*)$-algebras: $N,$ $ C\oplus N$, $D\oplus N$, $\mathcal{G}_{2,\tau}^{1,1} \oplus N$,  $C_{3,*}\oplus N$, $C_{3}^\mathsf{h}\oplus N$ or $C_{3,*}^\mathsf{h}\oplus N$, for some $D\in \mathcal{D}\cup \mathcal{D}_1$ and $h\in G$ with $|h|=2$.
			
			\item[3)] $l_{n}^{\sharp}(A)\leq 3,$ for $n$ large enough.
		\end{enumerate}
	\end{theorem}
	
	\begin{proof} Since any $(G,*)$-algebra in $\mathcal{I}$ has $(G,*)$-colength at least $4$ then  item $3)$ implies $1)$. Moreover, we have already notice that item $2)$ implies $3)$.

   In order to prove that  item  $1)$ implies   $2)$, we suppose that for all $Q\in\mathcal{I}$ we have $Q\notin \textnormal{var}^{\sharp}(A)$. By Remark \ref{12} we have $(FC_2)_{*}, \,(FC_2)^{\sharp}$, $FC_p$, $M_{g,\rho}$ $\notin \textnormal{var}^{\sharp}(A)$,  for all prime $p$ dividing $|G|$ and for all $g \in G$. Hence, by Theorem \ref{1}, $A$ has polynomial growth and so, by Theorem \ref{1} we have $A \sim_{T_G^*} B_1 \oplus \cdots \oplus B_m$, where $B_i$ is a finite-dimensional $(G, *)$-algebra over $F$ with either $B_i$ is nilpotent or $B_i\cong F+J(B_i)$, for $i=1,\ldots, m.$ If $B_i$ is nilpotent for all $i$ then $A\sim_{T_G^*}N$, where $N$ denotes a nilpotent \Gstar algebra. Therefore, we may assume that there exists $1\leq i\leq m$ such that $B_i=F+J(B_i),$ where $J(B_i)= J_{00}+J_{10}+J_{01}+J_{11}$ as given in (\ref{6}).
		
		Since $Q\notin \textnormal{var}^{\sharp}(B_i)$, for all $Q\in \mathcal{I}_3$, by Lemma \ref{3} we have $J_{10}=J_{01}=\{0\}$. Thus, we conclude that $B_i=B\oplus J_{00}$, where $B=F+J_{11}$. As long as $Q\notin \textnormal{var}^{\sharp}(B)$, for all $Q\in \mathcal{I}_4$, by Lemma \ref{306} we may assume $$B=\left( \underset{|g|=2}{\bigoplus } (F+(J_{11})_1+(J_{11})_g)\right)\oplus \left(\underset{|h|>2}{\bigoplus }{ (F +(J_{11})_1^+ +(J_{11})_h)}\right).$$
		
		Consider $g\in G$ with $|g|=2$. In this case, $F+(J_{11})_1+(J_{11})_g$ can be seen as a $(G,*)$-algebra with induced $\mathbb{Z}_2$-grading. Since $Q\notin \textnormal{var}^{\sharp}(B)$, for all $Q\in \mathcal{I}_2$, it follows from Theorem \ref{338} that $F+(J_{11})_1+(J_{11})_g$ is $T_G^*$-equivalent to one of the following $(G,*)$-algebras:
		\begin{center}
			$N,\, C\oplus N, C_{3,*}\oplus N , C_{3}^\mathsf{g} \oplus N, C_{3,*}^\mathsf{g}\oplus N$, $\mathcal{G}_{2,\tau}^{1,1} \oplus N, D_1\oplus N, D_1\oplus D_2\oplus  N$, \end{center} where $D_i\in \mathcal{D}^g, i\in \{1,2\},$ with $D_1\neq D_2.$ 
		
		Recall that if $h\in G$ with $|h|>2$ then, by Lemma \ref{10}, we have $x_{1,h}x_{2,h}\equiv 0$ on $B$, for all $x_{i,h}\in \{x_{i,h}^+,x_{i,h}^-\}$. Hence, $F +(J_{11})_1^+ +(J_{11})_h \in \textnormal{var}^{\sharp}(M_{h,\rho})$. Since $N_{3}^{h}$, $U_{3}^{h}$, $A_{2}^{h}\notin \textnormal{var}^{\sharp}(B)$, it follows from Theorem \ref{31} that $F +(J_{11})_1^+ +(J_{11})_h$ is $T_G^*$-equivalent to:
		$$N, C\oplus N, C_{2}^{\mathsf{h}}\oplus N, C_{2,*}^{\mathsf{h}}\oplus N\mbox{ or } C_{2}^{\mathsf{h}}\oplus C_{2,*}^{\mathsf{h}}\oplus N.$$ 
		
		Finally, we may construct the algebra $B$ by considering the direct sum of the previous $(G,*)$-algebras. Recalling that $Q\notin \textnormal{var}^{\sharp}(B)$, for all $Q\in \mathcal{I}_5$, then $B$ is $T_G^*$-equivalent to one of the following \Gstar algebras: 
		$$N,\, C\oplus N,C_{3,*}\oplus N, \mathcal{G}_{2,\tau}^{1,1} \oplus N, D_1\oplus N,D_1\oplus D_2\oplus N,C_{3}^\mathsf{g}\oplus N, C_{3,*}^\mathsf{g}\oplus N  $$ for some $g\in G$ with $|g|=2$ and $D_i\in \mathcal{D}, i\in \{1,2\}$, with $D_1\neq D_2.$ Since $A$ is $T_G^*$-equivalent to a direct sum of $(G,*)$-algebras of the type $B$ with nilpotent $(G,*)$-algebras then the item $2)$ follows.
	\end{proof}

 Now, we consider the subset $\mathcal{P}$ of $\mathcal{I}$ consisting of the $(G,*)$-algebras having $n$th $(G,*)$-colength $4$, for $n$ large enough, that is, 
$$\mathcal{P}= \{ \mathcal{G}_{3,\tau}^{1,1}, \mathcal{G}_{2,\imath_1}^{1,g}, \mathcal{G}_{2,\imath}^{h,h}, \mathcal{G}_{2,\imath_2}^{p,p}, \mathcal{G}_{2,\gamma}^{p,p^{-1}}, C_{4}^{\mathsf{h}},C_{4,*}^{\mathsf{h}}, C_{4,*}, C_3^{\mathsf{p}},  C_{3,*}^{\mathsf{p}}, S_1 , C_{3,*}\oplus K, \mathcal{G}_{2,\tau}^{1,1}\oplus K, \mathcal{G}_{2,\tau}^{1,1}\oplus C_{3,*}, $$ $$C_3^{\mathsf{h}}\oplus S_2, C_{3,*}^{\mathsf{h}}\oplus S_3, U_3^*,  N_3^*, W_{\nu_i}^{p,p^{-1}}, W_{\nu_j}^{1,p},  W_{\nu_3}^{h,h} \},$$ where $g,h,p\in G- \{1\} , |h|=2, |p|>2$, $\imath = \{\tau, \gamma, \psi\}$, $\imath_1\in \imath - \{\psi\},\imath_2\in \imath - \{\gamma\}$, $S_1\in \mathcal{D}_2, $ $K\in \mathcal{D}- \{C_{2,*}\}$, $S_2\in \mathcal{D}- \{C_2^{\mathsf{h}}\}$, $S_3\in \mathcal{D}- \{C_{2,*}^{\mathsf{h}}\}$, $i\in \{1,2,3\}$ and $j=\{2,3\}$.

As a consequence of the previous theorem, we can present the classification of the $(G,*)$-varieties of $(G,*)$-colength equal to $4.$

 \begin{corollary}\label{coro62}
Let $G$ be a finite abelian group of even order and $A$ be a finite-dimensional $(G,*)$-algebra. Then
		$l_{n}^{\sharp}(A)= 4,$ for $n$ large enough, if and only if $A\sim_{T_G^*} P\oplus N$, for some $P\in \mathcal{P}.$
 \end{corollary}

\begin{proof}
    It is enough to prove that 1) implies 2) so, suppose that there exists a constant $k>0 $ such that $l_n^{\sharp}(A)= 4$, for all $n\geq k$. In this case, if $B$ is any $(G, \ast)$-algebra in the set $\mathcal{I}$ defined in (\ref{ConjuntoI}) such that $l_n^{\sharp}(B)\geq 5$ then $B$ does not belong to $\textnormal{var}^{\sharp}(A).$ Moreover, since $l_n^{\sharp}(A)= 4$, $n\geq k$, by  Theorem \ref{157}, we have that $P\in \textnormal{var}^{\sharp}(A)$, for some $P\in \mathcal{P}$. Therefore, $P_n^\sharp  \cap \textnormal{Id}^\sharp(A) \subseteq P_n^\sharp \cap \textnormal{Id}^\sharp(P),$ for all $n.$

  It is known that there exists a nilpotent \((G, * )\)-algebra \( N \) whose \( T_G^* \)-ideal of identities is given by
\[
I = \left\langle P_l^\sharp  \cap \textnormal{Id}^\sharp(A),\ x_{1,g_{i_1}}^{\epsilon_1} \cdots x_{k,{g_{i_k}}}^{\epsilon_k} \ \middle| \ 1 \leq l \leq k-1, \epsilon_i\in\{+,-\}, g_{i_j}\in G\right\rangle_{T_G^*}.
\]

By definition of $I$, it is clear that
\[
P_n^\sharp \cap \textnormal{Id}^\sharp(A) \subseteq P_n^\sharp \cap \textnormal{Id}^\sharp(P) \cap \textnormal{Id}^\sharp(N) = P_n^\sharp \cap \textnormal{Id}^\sharp(P \oplus N), \text{ for all } n.
\]

Furthermore, we observe that any multilinear identity of degree \( n < k \) of \( N \) is a consequence of the identities \( P_l^\sharp \cap \textnormal{Id}^\sharp(A) \), for all \( 1 \leq l \leq k-1 \). Consequently,
\[
P_n^\sharp \cap \textnormal{Id}^\sharp(P) \cap \textnormal{Id}^\sharp(N) \subseteq P_n^\sharp \cap \textnormal{Id}^\sharp(N) \subseteq P_n^\sharp \cap \textnormal{Id}^\sharp(A), \text{ for all } n < k.
\]

Now, we observe that for \( n \geq k \) we have $4 = l_n^\sharp(P) \leq l_n^\sharp(A) = 4$, and so $P_n^\sharp \cap \textnormal{Id}^\sharp(P) = P_n^\sharp \cap \textnormal{Id}^\sharp(A).$ Since \( N \) is nilpotent of index \( k \), we obtain
$P_n^\sharp \cap \textnormal{Id}^\sharp(A) = P_n^\sharp \cap \textnormal{Id}^\sharp(P \oplus N)$, for all $n \geq k$. Therefore, $P_n^\sharp \cap \textnormal{Id}^*(A) = P_n^\sharp \cap \textnormal{Id}^\sharp(P \oplus N)$, for all $n$, and hence \( A \sim_{T_G^*} P \oplus N. \)
\end{proof}

The previous ideas can also be applied in the case where $|G|$ is odd. In this case, we may consider the set ${\mathcal{L}}$ consisting of all $(G,*)$-algebras defined in $\mathcal{I}$ that do not depend on the existence of an element of order $2$ in the group. Therefore, we have an analogous version of Lemma \ref{306}, where $A$ is $T_G^*$-equivalent to a $(G,*)$-algebra of type $B=(F+(J_{11})_1)\oplus (\underset{g\in G- \{1\}}{\bigoplus B_g}) $, where $B_g=F +(J_{11})_1^+ +(J_{11})_g \in \textnormal{var}^{\sharp}(M_{g,\rho})$ and $F+(J_{11})_1$ is a $(G,*)$-algebra with trivial $G$-grading. Thus, following the steps of the proof of the previous theorem, we may conclude the result below.

	\begin{theorem} \label{177} Let $G$ be a finite abelian group of odd order and $A$ be a finite-dimensional $(G,*)$-algebra. The following conditions are equivalent
		\begin{enumerate}
			\item[1)] $Q\notin \textnormal{var}^{\sharp}(A)$, for all $Q\in \mathcal{L}.$
			\item[2)] $A$ is $T_G^*$-equivalent to one of the following $(G,*)$-algebras: $N, C\oplus N,C_{3,*}\oplus N$, $\mathcal{G}_{2,\tau}^{1,1} \oplus N$, $D_1\oplus N$ or $D_1\oplus D_2\oplus N,$    for some $D_i\in \mathcal{D}, i\in \{1,2\},$ with $D_1\neq D_2.$
			\item[3)] $l_{n}^{\sharp}(A)\leq 3,$ for $n$ large enough.
		\end{enumerate}
	\end{theorem}

In what follows we consider the subset $\mathcal{R}$ of $\mathcal{L}$ consisting by all algebras having $(G,*)$-colength $4$, that is,
$$\mathcal{R}= \{ \mathcal{G}_{3,\tau}^{1,1}, \mathcal{G}_{2,\imath_1}^{1,g},\mathcal{G}_{2,\imath_2}^{g,g}, \mathcal{G}_{2,\gamma}^{g,g^{-1}}, C_{3}^{\mathsf{g}},C_{3,*}^{\mathsf{g}}, C_{4,*}, S_1 , C_{3,*}\oplus K, \mathcal{G}_{2,\tau}^{1,1}\oplus K,$$ $$ \mathcal{G}_{2,\tau}^{1,1}\oplus C_{3,*},  U_3^*,  N_3^*, W_{\nu_i}^{g,g^{-1}}, W_{\nu_j}^{1,g} \},$$ where $g\in G- \{1\}$, $\imath = \{\tau, \gamma, \psi\}$, $\imath_1\in \imath - \{\psi\},\imath_2\in \imath - \{\gamma\}$, $S_1\in \mathcal{D}_2, $ $K\in \mathcal{D}- \{C_{2,*}\}$, $i\in \{1,2,3\}$ and $j=\{2,3\}$.

As a consequence of the previous result we have the following corollary, which can be proved in the same way as in Corollary \ref{coro62}.

 \begin{corollary}\label{coro64}
Let $G$ be a finite abelian group of odd order and $A$ be a finite-dimensional $(G,*)$-algebra.     Then
		 $l_{n}^{\sharp}(A)= 4,$ for $n$ large enough if and only if $A\sim_{T_G^*} R\oplus N$, for some $R\in \mathcal{R}.$
 \end{corollary}

\end{document}